\documentclass[centertags, reqno]{amsart}                        
\usepackage{amssymb}
\usepackage{graphicx}
\usepackage{url}    
\usepackage{dsfont}

\newtheorem{thm}{Theorem}[section]
\newtheorem*{thm*}{Theorem} 
\newtheorem{lemma}[thm]{Lemma}
\newtheorem*{lem*}{Lemma}
\newtheorem{cor}[thm]{Corollary}
\newtheorem{prop}[thm]{Proposition}

\theoremstyle{remark}
\newtheorem{rem}[thm]{Remark} 
\newtheorem{ex}[thm]{Example}

\newcommand{\mr}{{\mathbb R}}

\newcommand{\mn}{{\mathbb N}}

\newcommand{\mc}{{\mathbb C}}
\newcommand{\md}{{\mathbb D}}

\newcommand{\mh}{{\mathbb H}}
\newcommand{\ms}{{\mathbb S}}

\renewcommand{\rho}{\varrho}
\newcommand{\eps}{\varepsilon}

\renewcommand{\Im}{\operatorname{Im}}
\renewcommand{\Re}{\operatorname{Re}}

 \newcommand{\dom}{\operatorname{Dom}} 

\newcommand{\dist}{\operatorname{dist}}

\newcommand{\hil}{\mathcal{H}}
\newcommand{\bdd}{\mathcal{B}}

 \usepackage{tikz}
\usepackage{xcolor}

\usepackage{mathtools}

\DeclareFontEncoding{FMS}{}{}
\DeclareFontSubstitution{FMS}{futm}{m}{n}
\DeclareFontEncoding{FMX}{}{}
\DeclareFontSubstitution{FMX}{futm}{m}{n}
\DeclareSymbolFont{fouriersymbols}{FMS}{futm}{m}{n}
\DeclareSymbolFont{fourierlargesymbols}{FMX}{futm}{m}{n}
\DeclareMathDelimiter{\VERT}{\mathord}{fouriersymbols}{152}{fourierlargesymbols}{147}

\usepackage{pgfplots}
  
\begin{document}

\title[$L_p$-spectrum for Schr\"odinger operators on the hyperbolic plane]{$L_p$-spectrum and Lieb-Thirring inequalities for Schr\"odinger operators on the hyperbolic plane} 

\author[M. Hansmann]{Marcel Hansmann}
\address{Faculty of Mathematics\\ 
Chemnitz University of Technology\\
Chemnitz\\
Germany.}
\email{marcel.hansmann@mathematik.tu-chemnitz.de}

\thanks{This work was funded by the Deutsche Forschungsgemeinschaft (DFG, German Research Foundation) - Project number HA 7732/2-1. I would like to thank Hendrik Vogt for pointing me to the references \cite{MR1673414, MR1744777}.}


\begin{abstract}
This paper deals with the $L_p$-spectrum of Schr\"odinger operators on the hyperbolic plane. We establish Lieb-Thirring type inequalities for discrete eigenvalues and study their dependence on $p$. Some bounds on individual eigenvalues are derived as well.  
\end{abstract}

\maketitle 

\section{Introduction}

The study of spectral properties of non-selfadjoint Schr\"odinger operators $-\Delta+V$ in $L_2(\mr^d)$, with a \emph{complex-valued} potential $V$, has attracted considerable attention in recent years. In particular, many works have been dedicated to the derivation of non-selfadjoint versions of the famous Lieb-Thirring inequalities (first considered by Lieb and Thirring for real-valued potentials in \cite{Lieb75,Lieb76}) and to the problem of finding good upper bounds on individual eigenvalues. Let us mention \cite{MR2260376, MR2559715, MR2540070, MR2596049, Hansmann11, MR3077277, MR3016473, GK15, MR3627408, MR3730931, frank15} as some references for the former topic and \cite{MR1819914, MR2540070, MR2651940, Frank11, FLS11, MR3451537, MR3713021, frank15} as some references for the latter.

While it is natural to study Schr\"odinger operators in the Hilbert space $L_2(\mr^d)$, there also exist good reasons (see e.g. \cite{MR670130}) to consider them in $L_p(\mr^d)$, for $p \neq 2,$ as well. However, from a spectral perspective this isn't interesting at all. Indeed, it has been shown in \cite{MR836002} that under weak assumptions on the potential $V$ the $L_p$-spectra of selfadjoint Schr\"odinger operators coincide. Moreover, later results showed that this is the case in the non-selfadjoint setting as well (see \cite{MR1673414, MR1744777}). Even more is true: the fact that the underlying manifold is $\mr^d$ doesn't play a role either. For instance, it was shown in \cite{sturm} that the $L_p$-spectra of the Laplace-Beltrami operator on a complete Riemannian manifold $M$ with Ricci-curvature bounded from below are $p$-independent, provided the volume of $M$ grows at most sub-exponentially.  

One of the simplest manifolds where the $L_p$-spectrum of the Laplace-Beltrami operator \emph{does} depend on $p$ is given by the hyperbolic plane $\mh$. In the half-space model this manifold is given by
\begin{equation*}
\mh=\{ (y,t) \in \mr^{2} : y \in \mr, 0 < t < \infty\},
\end{equation*}
together with  the conformal metric $ ds^2 =t^{-2} (dy^2+dt^2)$. It has been shown in \cite{MR937635} that the spectrum of  $-\Delta_p$ in $L_p(\mh), 1 \leq p < \infty,$ consists of the parabolic sets 
\begin{equation}
  \label{eq:1}
  \Sigma_p:=\left\{ a+ib : a \geq 1/(pp') \text{ and } b^2 \leq (1-2/p)^2\left(a- {1}/(pp')\right) \right\},
\end{equation}
see Figure \ref{fig:1}. Here $p'$ denotes the conjugate exponent, i.e.
\begin{equation*}
1/p +1/p'=1. 
\end{equation*}
In particular, the spectrum of the selfadjoint operator $-\Delta_2$ consists of the interval $\Sigma_2=[1/4,\infty)$ and in case $p \neq 2$ the spectrum is the set of points on and inside the parabola with vertex $\lambda=1/(pp')$ and focus $\lambda=1/4$. We see that $\Sigma_p=\Sigma_{p'}$, reflecting the fact that (up to a reflection on the real line) the spectra of $-\Delta_p$ and its adjoint $-\Delta_{p'}$ coincide. Moreover, let us remark that, while in case $p=2$ the spectrum is clearly purely essential, it seems to be unknown whether the same is true for $p \neq 2$ as well (we conjecture that it is).

 \begin{figure}[h]
   \centering 
  \includegraphics[width=0.55\textwidth]{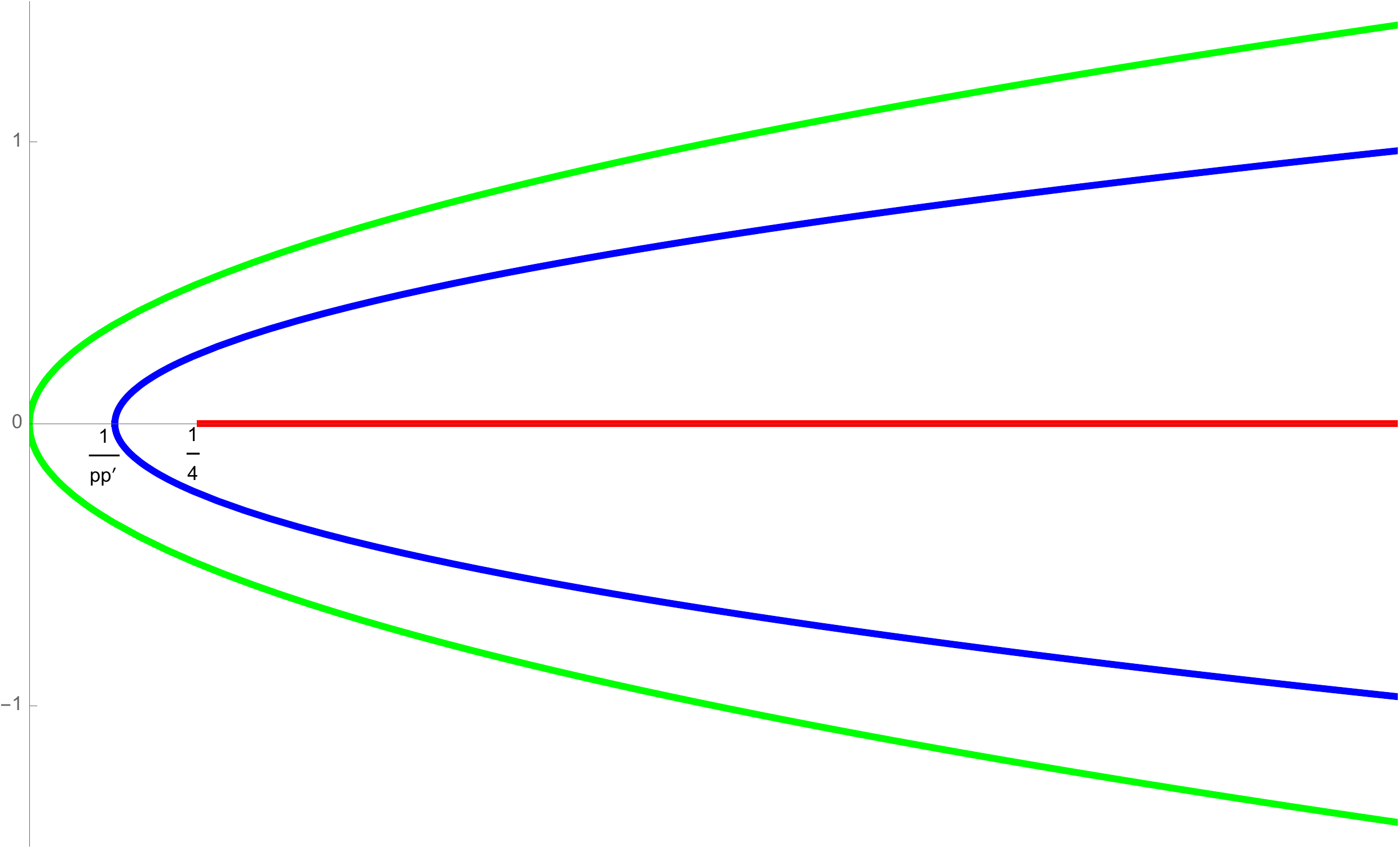}   
\caption{{\small The boundaries of $\Sigma_p=\sigma(-\Delta_p)$ drawn for $p=1$, some general $p \in (1,2)$ and $p=2$, respectively (from outer to inner).}}
\label{fig:1}
 \end{figure}

In the present paper we will study the Schr\"odinger operator 
\begin{equation}
H_p=-\Delta_p+V \quad \text{in} \quad L_p(\mh), \quad 1 < p < \infty,\label{eq:2}
\end{equation}
given the assumption that 
\begin{equation}
V \in L_r(\mh) \quad \text{for some} \quad r \geq \max(p,p').\label{eq:3}
\end{equation}
We will see below that in this case the operator of multiplication by $V$ is $-\Delta_p$-compact and hence the essential spectra of $H_p$ and $-\Delta_p$ coincide. In particular, the topological boundary $\partial \Sigma_p$, not containing any isolated points, belongs to the essential spectrum of both operators. While the essential spectrum is stable, other parts of the spectrum of $-\Delta_p$ will change with the introduction of the perturbation $V$. In particular, the spectrum of $H_p$ in $\Sigma_p^c$ (the resolvent set of $-\Delta_p$) will consist of an at most countable number of discrete eigenvalues, which can accumulate at $\partial \Sigma_p$ only. It is our aim to say more about the speed of this accumulation, and its dependence on $p$, by deriving suitable Lieb-Thirring type inequalities. In addition, we will also provide some estimates on individual eigenvalues. 

As far as we can say, this paper constitutes the first work on such topics in a non-Hilbert space context. Moreover, we think that our results are even new in the Hilbert space case $p=2$, where the only existing articles we are aware of are \cite{MR1454250} and \cite{MR2639180}, respectively. Here \cite{MR1454250} considers the selfadjoint case only and provides bounds on the number of discrete eigenvalues of $-\Delta_2+V$ in hyperbolic space of dimension $d \geq 3$, whereas the abstract results of \cite{MR2639180} also apply to complex-valued potentials and could, in principle, be used to obtain some estimates on the discrete eigenvalues of $H_2$ in the half-plane $\{\lambda \in \mc : \Re(\lambda) < 1/4\}$. In contrast to this, the results we will derive in this paper will provide information on \emph{all} discrete eigenvalues of $H_p$ in $\Sigma_p^c$. 
   
While in the present paper we restrict ourselves to the two-dimensional hyperbolic plane, let us at least mention that in principle we can obtain results for higher dimensional hyperbolic space as well. Indeed, our results rely on the explicit knowledge of the green kernel of $-\Delta_p$, which is available in all dimensions (though it gets more complicated in case $d \geq 4$).

\section{Main results}

This section contains the main results of this paper. We use some standard terminology concerning operators and spectra, which is reviewed in Appendix \ref{ap:oper}.

\subsection{Bounds on eigenvalues} \label{sec:21}

We begin with two results concerning the location of the discrete spectrum $\sigma_d(H_p)$, starting with the case $p=2$.

\begin{thm}\label{thm:1}
  Let $2 \leq r < \infty$ and $V \in L_r(\mh)$. If $\lambda \in \sigma_d(H_2) $, then
  \begin{equation}
    \label{eq:4}
    \dist(\lambda,[1/4, \infty))^{(r-1)} \left( 1+ \frac 1 {2|1/4-\lambda|^{1/2}} \right) \leq 2^{3/2} C_0 \|V\|_r^r,
  \end{equation}
where 
\begin{equation}
  \label{eq:5}
  C_0=\frac{1 + \frac \pi 2 \coth(\frac \pi 2 )}{4\pi} \approx 0.216
\end{equation}
\end{thm}
In particular, we see that (\ref{eq:4}) implies that the distance of the discrete eigenvalues to the essential spectrum is bounded above, i.e. for $\lambda \in \sigma_d(H_2)$ we have
\begin{equation*}
       \dist(\lambda,[1/ 4, \infty)) \leq 2^{3/(2(r-1))} C_0^{1/(r-1)} \|V\|_r^{r/(r-1)}.
\end{equation*} 
\begin{rem}
Given the same assumptions on $V$, for the Schr\"odinger operator $-\Delta+V$ in $L_2(\mr^2)$ it is even known that the imaginary part of a discrete eigenvalue needs to be small if its real part is large,  see \cite{frank15}. Whether a similar statement remains true on the hyperpolic plane is an open question.
\end{rem} 
In case $p\neq 2$, the result we obtain is more complicated. For its statement, it is convenient to introduce $\gamma_p \in [0,1/2]$ by setting
\begin{equation}
  \label{eq:7}
  \gamma_p:= \frac 1 2 \left|1-\frac 2 p\right|.
\end{equation}
A short computation shows that $\gamma_p^2=1/4-1/(pp')$, which is the focal length of the parabola $\Sigma_p=\sigma(-\Delta_p)$ (the distance between focus and vertex). In particular, we see that $\gamma_2=0$ and $\gamma_p=\gamma_{p'}$. 

\begin{thm}\label{thm:2}
Let $2 < \max(p,p') \leq r < \infty$ and suppose that $V \in L_r(\mh)$. If $\lambda \in \sigma_d(H_p) \cap \Sigma_p^c$, then
  \begin{equation}
      \label{eq:6}
\left( \frac{\dist(\lambda,\Sigma_p)}{|1/4-\lambda|^{1/2}}\right)^{2r-2}\left( 1 + \frac{|1/4-\lambda|^{1/2}} {8 \dist(\lambda,\Sigma_p)} \right)^{2r\gamma_p+1} \leq 16^{2r-2} C_0 \|V\|_r^r,
  \end{equation}
where $C_0$ is as defined in (\ref{eq:5}).
\end{thm}
Since $\gamma_p=\gamma_{p'}$ and $\Sigma_p=\Sigma_{p'}$ we see that Theorem \ref{thm:2} provides the same bounds for the eigenvalues of $H_p$ and $H_{p'}$, respectively. This  is not a coincidence but follows from the fact that $H_p^*=H_{p'}$ and hence (up to a reflection on the real line) the spectra and discrete spectra of $H_p$ and $H_{p'}$ coincide. This will be proved in Proposition \ref{prop:1} below.  The same phenomenon will be observed in other results of this paper.  
\begin{rem}
  We note that the term $|\lambda - 1/4|$ in (\ref{eq:6}) does not play the same role as in (\ref{eq:4}), since in case $p\neq 2$ the point $1/4$ is in the interior of the spectrum. 
\end{rem}
While (\ref{eq:6}) puts some restrictions on the location of the discrete eigenvalues, we emphasize that in contrast to the case $p=2$, in case $p\neq 2$ the bound (\ref{eq:6}) does not imply that $\dist(\lambda,\Sigma_p)$ is bounded above for $\lambda \in \sigma_d(H_p) \cap \Sigma_p^c$. We do not know whether this reflects a real difference between the two cases, or whether it is just an artefact of our method of proof.

\subsection{Lieb-Thirring inequalities}  \label{sec:22}

We now consider the speed of accumulation of discrete eigenvalues, again starting with the Hilbert space case $p=2$. In the following estimate we distinguish between discrete eigenvalues lying in a disk around $1/4$ (with radius depending on $V$) and eigenvalues lying outside this disk.

\begin{thm}[$p=2$]\label{thm:3}
Let $2 \leq r < \infty$ and $V \in L_r(\mh)$ . Let $(\lambda_j)$ denote an enumeration of the discrete eigenvalues of $H_2$, each eigenvalue being counted according to its algebraic multiplicity. Then for every $\tau \in (0,1)$ there exist constants
$C$ and $C'$, both depending on $\tau$ and $r$, such that the following holds:
\begin{enumerate}
    \item[(ia)] If $2 \leq r \leq 3-\tau$, then
$$\sum_{|1/4-\lambda_j|^{r-3/2} \leq (2\|V\|_r)^{r}} \left( \frac{\dist(\lambda_j,[1/4,\infty))}{|1/4 -\lambda_j|^{1/2}}\right)^{r+\tau} \leq C\cdot \|V\|_r^{\frac r {2r-3}(r+\tau)}.$$
\item[(ib)] If $r>3-\tau$, then
$$\sum_{|1/4-\lambda_j|^{r-3/2} \leq (2\|V\|_r)^{r}} \frac{\dist(\lambda_j,[1/4,\infty))^{r+\tau}}{|1/4 -\lambda_j|^{3/2}} \leq C\cdot \|V\|_r^{\frac r {2r-3}(2r-3+2\tau)}.$$
\item[(ii)] \begin{eqnarray*}
\sum_{|1/4-\lambda_j|^{r-3/2} >(2\|V\|_r)^{r} } \frac{\dist(\lambda_j,[1/4,\infty))^{r+\tau}}{|1/4 -\lambda_j|^{\frac{3+3\tau} 2}}
\leq C' \cdot  \|V\|_r^{\frac r {2r-3}(2r-3-\tau)}.
\end{eqnarray*}
\end{enumerate}
\end{thm}

\begin{rem}
The previous theorem has consequences for sequences $(E_j)$ of discrete eigenvalues converging to some $E \in [1/4,\infty)$. For instance, 
\begin{enumerate}
    \item[-] if $E > 1/4$, then $(\Im E_j) \in l_{r+\tau}$,
    \item[-] if $E=1/4$ and $\Re(E_j) \leq 1/4$, then  $(|E_j-1/4|) \in l_q,$ where
  \begin{equation}
    \label{eq:56}
q = \left\{
  \begin{array}{cl}
    (r+\tau)/2, & \text{ if } 2 \leq r \leq 3-\tau, \\
    r-3/2+\tau, & \text{ if } r > 3-\tau.
  \end{array}\right.
  \end{equation}
\end{enumerate}
In particular, concerning sequences of eigenvalues converging to the bottom of the essential spectrum we obtain different results for $r<3$ and $r > 3$, respectively. Whether this reflects a real difference between the two cases is another interesting open question.
\end{rem}

For the next result in case $p \neq 2$ we again recall that $1/(pp')$ is the vertex of $\Sigma_p$.

\begin{thm}[$p \neq 2$]\label{thm:4}
Let $2 < \max(p,p') \leq r < \infty$ and $V \in L_r(\mh)$ . Let $(\lambda_j)$ denote an enumeration of the discrete eigenvalues of $H_p$ in $\Sigma_p^c$, each eigenvalue being counted according to its algebraic multiplicity. Moreover, set
$$ k:=k(r,p):=r(2-2\gamma_p)-2 \quad (\in (2,\infty)),$$
where $\gamma_p$ is as defined in (\ref{eq:7}). Then for every $\tau \in (0,1)$ there exist $0< \eps_1,\eps_2,\eps_3< 4\tau$ and constants $C,C'$, depending on $\tau, r$ and $p$, such that the following holds:
\begin{enumerate} {\small
      \item[(i)]
\begin{eqnarray*}
\sum_{|\frac 1 {pp'}-\lambda_j|^{(k-1)/2} \leq (2\|V\|_r)^{{r}}}  \dist(\lambda_j,\Sigma_p)^{k+\eps_1} 
\leq C \|V\|_r^r \left(\|V\|_r^{\frac r {k-1}}+\gamma_p\right)^{k+1+\eps_2}.
\end{eqnarray*}
}
\item[(ii)] {\small
\begin{eqnarray*} 
  \sum_{|\frac 1 {pp'}-\lambda_j|^{(k-1)/2} > (2\|V\|_r)^{r}}  \frac{\dist(\lambda,\Sigma_p)^{k+\eps_1} }{(|\lambda-\frac 1 {pp'}|^{1/2}+2\gamma_p)^{k+1+\eps_3}} \leq C' \|V\|_r^r \left(\|V\|_r^{\frac r {k-1}}+\gamma_p\right)^{-\tau}.
\end{eqnarray*}
} 
\end{enumerate}
\end{thm}

\begin{rem}
We see that in contrast to the case $p=2$ (where the parabola $\Sigma_p$ degenerates to an interval) here we obtain the same information on all sequences of eigenvalues, independent of the fact whether they converge to the vertex $1/(pp')$ or to a generic point of the boundary of $\Sigma_p$. Still, also here differentiating between 'small' and 'large' eigenvalues has its value, since the estimate in (ii) also provides information on sequences $(E_j)$ of eigenvalues diverging to $\infty$.
\end{rem}

To see how the above estimates depend on $p$, let us assume that $V \in L_r( \mh)$ for some fixed $r > 2$ and let (without restriction) $2<p\leq r$. Suppose that $(E_j)$ is a sequence of discrete eigenvalues of $H_p=-\Delta_p+V$ converging to some $E \in \partial \Sigma_p$. Then $(\dist(E_j,\Sigma_p)) \in l_{k+\eps_1}$, where
$$ k = r(2-2\gamma_p)-2 = r(1+ 2/p)-2.$$
In particular, we see that $k$ decreases for increasing $p$. This can be interpreted as saying that the constraints on sequences of eigenvalues of $H_p, 2 < p \leq r,$ are getting more severe with increasing $p$ and are maximal for $p=r$, in which case $(\dist(E_j,\Sigma_p)) \in l_{r+\eps_1}$ (just like in the Hilbert space case). 

Finally, let us emphasize that the results of Theorem \ref{thm:3} and \ref{thm:4} are not 'continuous' in $p$, but have a 'discontinuity' at $p=2$. To see this let $(E_j)$ again denote a sequence of eigenvalues of $H_p$, converging to some $E \in \partial \Sigma_p \setminus \{1/(pp')\}$. Then in case $p=2$ the sequence $(\dist(E_j,\Sigma_p))$ is 'almost' in $l_r$, while in case $p=2+\eps$ (with $\eps$ sufficiently small) it is only 'almost' in $l_{2r-2}$.  Since $2r-2 > r$ for $r>2$, the latter result is weaker than the former. Whether this discontinuity corresponds to a real phenomenon seems like a further interesting question for future research. 

\subsection{On proofs and how the paper proceeds} The results in Section \ref{sec:21} will be proved using the Birman-Schwinger principle, which requires us to obtain good upper bounds on the norm of the Birman-Schwinger operator $V(-\Delta_p-\lambda)^{-1}$. We will obtain such bounds via corresponding Schatten-von Neumann norm estimates (in case $p=2$) and summing norm estimates (in case $p \neq 2$), respectively. 

The Lieb-Thirring estimates in Section \ref{sec:22} will be obtained using a method first introduced in \cite{MR2413204} and \cite{MR2481997}: We will construct suitable holomorphic functions (perturbation determinants) whose zeros coincide with the eigenvalues of $H_p$ and we will then use a complex analysis result of Borichev, Golinskii and Kupin \cite{MR2481997} to study these zeros. While this method has been applied in many different cases for operators in Hilbert spaces (see the citations at the beginning of the introduction), we seem to be the first to apply it in a general Banach space context. In order to make this work we will rely on a general theory of perturbation determinants in Banach spaces recently obtained in \cite{Hansmann15}. 

The paper will proceed as follows: In the next section we will provide the precise definitions of $-\Delta_p$ and $H_p$, respectively, and we will derive and recall some of their properties. In Section \ref{sec:estimates} we will derive various norm estimates on the resolvent of $-\Delta_p$ and on the Birman-Schwinger operator $V(-\Delta_p-\lambda)^{-1}$. These results will be used in Section \ref{sec:5} to prove the results of Section \ref{sec:21}. In Section \ref{sec:6} we will derive an abstract Lieb-Thirring estimate, which will be applied in Section \ref{sec:7} to prove the results of Section \ref{sec:22}. The paper is concluded by an appendix with three parts: in the first part we recall some standard results concerning operators and their spectra; in the second part we review the theory of perturbation determinants in Banach spaces and we introduce the Schatten-von Neumann and $(p,q)$-summing ideals; finally, in the third part we recall some results from complex interpolation theory which are required in Section \ref{sec:estimates}.

\section{The operators}\label{sec:operators}

Some standard results (and terminology) for operators and spectra used throughout this section are compiled in Appendix \ref{ap:oper}.

\subsection{The Hyperbolic plane, its Laplace-Beltrami operator and Green's function}
  
Most of the material in this section is taken from  \cite[Section 5.7]{MR990239} (see also \cite{MR2344504}).

(a) As noted in the introduction, in the half-space model the hyperbolic plane is described by
$$\mh=\{ x=(y,t) \in \mr^{2} : y \in \mr, 0 < t < \infty\}.$$
Equipped with the conformal metric $ ds^2 =t^{-2} (dy^2+dt^2)$ it is a complete Riemannian manifold with volume element
$$ d\mu(x):= t^{-2} dy\: dt.$$
The Riemannian distance $d(x,x')$ between two points $x=(y,t),x'=(y',t')$ can be computed via  the identity 
\begin{equation*}
{\cosh(d(x,x'))+1}= \frac{|y-y'|^2+(t+t')^2}{2tt'}.
\end{equation*}
It is sometimes convenient to use so-called \emph{geodesic polar coordinates}, see, e.g., \cite[Section 3.1]{MR3100414}: We fix $x_0 \in \mh$ and identify $x\in \mh
  \setminus \{x_0\}$ with the pair
$$ (r, \xi) \in (0,\infty) \times \ms^1,$$
where $r:=d(x,x_0)$ and $\xi \in
\ms^1$ denotes the unit vector at $x_0$ which is tangent to the geodesic
ray that starts at $x_0$ and contains $x$ (here we identify the unit
tangent space at $x_0$ with the sphere $\ms^1$). The volume element in
geodesic polar coordinates is given by
$$ \sinh(r) dr \:d\xi,$$ with $d\xi$ denoting the surface measure on $\ms^1$.

(b) The Laplace-Beltrami operator on $\mh$ is given by
$$ -\Delta = - t^2(\partial^2/\partial y^2 +\partial^2/\partial t^2).$$
It is  essentially selfadjoint on $C_c^\infty(\mh)$ and so its closure (also denoted by $-\Delta$) is selfadjoint in $L_2(\mh)$. Since $-\Delta$ is positive, $\Delta$ generates a contraction semigroup $e^{t\Delta}$ on $L^2(\mh),$ which can be shown to be submarkovian (i.e. it is positivity preserving and a contraction on $L_\infty(\mh)$). In particular, this implies that $e^{t\Delta}$ maps $L_1(\mh) \cap L_\infty(\mh)$ into itself and $e^{t\Delta}|_{L_1\cap L_\infty}$ can be extended to a submarkovian semigroup $T_p(t)$ on $L_p(\mh)$ for every $p \in [1,\infty]$. Moreover, these semigroups are consistent, i.e. $T_p(t)|_{L_p \cap L_q}=T_q(t)|_{L_p\cap L_q}$ for $p\neq q$. In case that $p \in [1,\infty)$ they are also strongly continuous. In the following, we denote the generator of $T_p(t), 1 \leq p < \infty,$ by $\Delta_p$ (in particular, $-\Delta = -\Delta_2$). Note that the domain of $-\Delta_p$ coincides with the Sobolev space $W_2^p(\mh)$, see, e.g., \cite{MR705991} and \cite[Section 7.4.5]{MR1163193}. Identifying the adjoint space of $L_p(\mh)$ with $L_{p'}(\mh)$, the adjoint of $-\Delta_p, 1<p<\infty,$ is equal to $-\Delta_{p'}$. The spectrum of $-\Delta_p$ is equal to the set $\Sigma_p$ defined in (\ref{eq:1}). Concerning the structure of the spectrum let us mention that for $p>2$ each point in the interior of $\Sigma_p$ is an eigenvalue, see \cite{taylor89}. 
\begin{rem}
In general it seems to be unknown whether $\sigma(-\Delta_p)$ is purely essential.  
\end{rem}

(c) For $\lambda \in \rho(-\Delta_p)=\Sigma_p^c$ the resolvent $(-\Delta_p-\lambda)^{-1}$ is an integral operator on $L_p(\mh)$ whose kernel (green function) $G_\lambda(x,x')$ depends on the Riemannian distance $d(x,x')$ only. In order to present an explicit formula for this kernel it is convenient to first map $\mc \setminus [1/4,\infty)$ conformally onto the half plane $\{ \lambda : \Re(\lambda)> 1/2\}$ by setting 
\begin{equation}
s:= s(\lambda)= {1}/2 + \sqrt{ {1}/4 - \lambda}, \qquad \lambda \in \mc \setminus \left[{1}/ 4,\infty\right),\label{eq:8}
\end{equation}
i.e. $\lambda=-s(s-1)$.
\begin{rem}
  We note that throughout this article we use the branch of the square
  root on $\mc \setminus (-\infty,0]$ which has positive real part.
\end{rem}
With $d=d(x,x')$ we have
\begin{equation}
  \label{eq:9}
  G_\lambda(x,x') = \frac{1}{2^{3/2} \pi} \int_d^\infty \frac{e^{-r(s-\frac 1 2)}}{(\cosh(r)-\cosh(d))^{1/2}} dr,
\end{equation}
see, e.g., \cite[Formula (2.13)]{MR937635}.

\subsection{The Schr\"odinger operator} \label{sec:32}

Let $p > 1$ and $V \in L_p(\mh) +L_\infty(\mh)$. We use the same symbol $V$ to denote the maximal operator of multiplication by $V$ in $L_p(\mh)$. The Sobolev embedding theorems, see e.g. \cite{MR2389638}, imply that $\dom(-\Delta_p)=W_2^p(\mh) \subset L_p(\hil) \cap L_\infty(\hil)$ and hence the Schr\"odinger operator 
$$ H_p= -\Delta_p + V$$
is well defined on $\dom(H_p) := \dom(-\Delta_p)$. We now assume that $V$ satisfies the stronger assumption (\ref{eq:3}), i.e. $V \in L_r(\mh)$ for some $r \geq \max(p,p')$. Then in case $p \geq 2$ we will prove in Theorem \ref{thm:6} below that $V$ is $-\Delta_p$-compact and hence $H_p$ is closed and $\sigma_{ess}(H_p)=\sigma_{ess}(-\Delta_p)$, see Appendix \ref{ap:oper} (b). 

The case $1<p<2$ can be reduced to the case $p > 2$ with the help of the following proposition. 
\begin{prop}\label{prop:1}
Let $1<p<2$ and suppose that $V \in L_r(\mh)$ for some $r\geq p'=p/(p-1)$. Then
$ H_p = H_{p'}^*.$ In particular, up to a reflection on the real line the essential and discrete spectra of $H_p$ and $H_{p'}$ coincide. 
\end{prop}
Note that some standard properties of the adjoint operator are reviewed in Appendix \ref{ap:oper} (c).
\begin{proof}[Proof of the proposition] 
Just for this proof let us write $V_p$ for the maximal operator of multiplication by $V$ in $L_p(\mh)$, so $V_p^*=V_{p'}$. Since $0 \in \rho(-\Delta_{p'})$ we then have
$$H_{p'} = (-\Delta_{p'}+ V_{p'})=(I+V_{p'}(-\Delta_{p'})^{-1})(-\Delta_{p'}).$$
Since $p'>2$ by the previous discussion (or Theorem \ref{thm:6}) the operator $(I+V_{p'}(-\Delta_{p'})^{-1})$ is bounded on $L_{p'}(\mh)$, so we obtain
\begin{eqnarray*}
  H_{p'}^* &=& (-\Delta_{p'})^*(I+V_{p'}(-\Delta_{p'})^{-1})^* = (-\Delta_{p})(I+[V_{p'}(-\Delta_{p'})^{-1}]^*).
\end{eqnarray*} 
Now general theory only allows us to conclude that $[V_{p'}(-\Delta_{p'})^{-1}]^* \supset (-\Delta_p)^{-1}V_p$. However, since the operator on the left-hand side of this inclusion is bounded on $L_{p}$ (even compact), it coincides with the closure of the operator on the right-hand side and hence
 $H_{p'}^* =  (-\Delta_{p})(I+\overline{(-\Delta_{p})^{-1}V_p})$. But here the domain of the product on the right is equal to $\dom(-\Delta_p)$ and on this set the operators $\overline{(-\Delta_{p})^{-1}V_p}$ and $(-\Delta_{p})^{-1}V_p$ coincide. So finally we see that 
$$H_{p'}^* = (-\Delta_{p})(I+(-\Delta_{p})^{-1}V_p) = -\Delta_p + V_p = H_p.$$
\end{proof}

\section{A variety of estimates}\label{sec:estimates}

In this section we derive various estimates on the resolvent and the resolvent kernel of $-\Delta_p$ and on the Birman-Schwinger operator $V(-\Delta_p-\lambda)^{-1}$. To this end, it will be necessary to first map the resolvent set
$\Sigma_p^c = \rho(-\Delta_p)$ conformally onto the right half-plane
\begin{equation*}
  \mc_+ := \{ \lambda \in \mc : \Re(\lambda) > 0\}. 
\end{equation*}
Since (\ref{eq:1}) shows that $\Sigma_p^c$ is just the set outside the parabola parameterized by $\mr \ni t \mapsto 1/(pp') + t^2+it(1-2/p)$, such a conformal map (or rather its inverse) is given by 
\begin{eqnarray} 
\Psi_p: \mc_+ &\to& \Sigma_p^c     \label{eq:11}\\
z &\mapsto& 1/(pp')-z^2-z|1-2/p|=:\lambda. \nonumber      
\end{eqnarray} 
Using $\gamma_p=1/2|1-2/p|$ as defined in (\ref{eq:7}) a short calculation shows that 
\begin{equation}
  \label{eq:13}
  \lambda=\Psi_p(z) = 1/4-(z+\gamma_p)^2
\end{equation}
and
\begin{equation}
z = \Psi_p^{-1}(\lambda) = - \gamma_p+\sqrt{1/4-\lambda}.\label{eq:14}
\end{equation}
 We note that with $s=s(\lambda)$ as defined in (\ref{eq:8}) we have 
\begin{equation}
  \label{eq:15}
s=\Psi_p^{-1}(\lambda)+ 1/2 +\gamma_p, \quad \lambda \in \Sigma_p^c.
\end{equation}
The following lemma will allow us to freely switch between estimates in terms of $\lambda$ and $z$, respectively.
\begin{lemma}\label{lem:41}
  Let $z \in \mc_+$ and $\lambda=\Psi_p(z), 1 \leq p < \infty$. 
  \begin{enumerate}
  \item[(i)] If $p=2$, then
    \begin{equation}
      \label{eq:16}
      |z| \Re(z) \leq \dist(\lambda,[1/4,\infty)) \leq 2 \Re(z) |z|.
    \end{equation}
    \item[(ii)] If $p \neq 2$, then
      \begin{equation}
      \frac{|z+\gamma_p| \Re(z)}{4} \leq \dist(\lambda,\Sigma_p) \leq 16 |z+\gamma_p| \Re(z).\label{eq:17}
    \end{equation}
  \end{enumerate}
\end{lemma}
\begin{proof}
(i) In case $p=2$ a short computation shows that with $\lambda = 1/4-z^2$:
$$ \dist(\lambda,[1/4,\infty)) = \left\{
  \begin{array}{cl}
    |z|^2, & |\Im(z)| \leq \Re(z), \\
    2 \Re(z) |\Im(z)|, & |\Im(z)| > \Re(z). \\
  \end{array}\right.$$
Since $|z| \geq \Re(z)$, and $\sqrt{2}|\Im(z)| > |z|$ if $|\Im(z)| > \Re(z)$, we obtain the lower bound in (\ref{eq:16}). Similarly, since $|z| \leq \sqrt{2} |\Re(z)|$ if $|\Im(z)| \leq \Re(z)$, and since $|\Im(z)| \leq |z|$, we obtain the upper bound as well.

(ii) In case $p \neq 2$ we proceed more indirectly. Let $\phi : \md \to \mc_+$ denote an arbitrary conformal mapping between the unit disk $\md=\{w \in \mc : |w| \leq 1\}$ and the right half-plane. Then the Koebe distortion theorem (see \cite{MR1217706}, page 9) implies that with $z = \phi(w):$  
\begin{equation}
    \label{eq:18}
\frac 1 4 |\phi'(w)| (1-|w|) \leq \Re(z)=\dist(z,\partial \mc_+) \leq 2 |\phi'(w)| (1-|w|).
\end{equation}
The function $f:= \Psi_p \circ \phi : \md \to \Sigma_p^c$ is conformal as well, so applying the distortion theorem a second time we obtain with $\lambda=\Psi_p(\phi(w))=\Psi_p(z)$:
\begin{equation}
  \label{eq:19}
\frac 1 4 |f'(w)| (1-|w|) \leq \dist(\lambda,\Sigma_p) \leq 2 |f'(w)| (1-|w|).  
\end{equation}
But $f'(w)= \Psi_p'(z)\cdot \phi'(w)$, so (\ref{eq:18}) and (\ref{eq:19}) together imply that
\begin{eqnarray*}
  \dist(\lambda,\Sigma_p) \leq 2 |\Psi_p'(z)| |\phi'(w)| (1-|w|) \leq 8 |\Psi_p'(z)| \Re(z).
\end{eqnarray*}
Since $\Psi_p'(z)=-2(z+\gamma_p)$, this proves the upper bound in (\ref{eq:17}). The lower bound is proved similarly.
\end{proof}

\subsection{Kernel estimates}

In the following we present a series of estimates on the green function $G_\lambda(.,.)$ defined in (\ref{eq:9}), starting with the following one due to Elstrodt. As above we write $\|.\|_p$ for the norm in $L_p(\mh)$.
\begin{prop}[\cite{MR0360473}, Corollary 7.3. (see also \cite{MR0360472})]
 For $\lambda \in \mc \setminus [\frac 1 4, \infty)$ and $s=s(\lambda)$ as defined in (\ref{eq:8}) the following holds:
 \begin{equation}
   \label{eq:20}
  \sup_{x \in \mh} \|G_\lambda(x,.)\|_{2}^2 \leq \left\{
    \begin{array}{cl}
\frac{|\Im(\psi(s))|} {2 \pi |\Im(\lambda)|},  & \lambda \in \mc \setminus \mr,  \\[4pt]
\frac{\psi'(s)} {4 \pi (s-\frac 1 2)}, & \lambda \in \mr.
    \end{array}\right.
\end{equation}
Here $\psi(s)=\frac d {ds} \ln(\Gamma(s))$ denotes the Digamma function.
\end{prop}
It is convenient to rewrite this estimate as follows.
\begin{cor}
 For all $z \in \mc_+$ and $\lambda=\Psi_2(z)=\frac 1 4 -z^2$ we have
 \begin{equation}
   \label{eq:21}
   \sup_{x \in \mh} \|G_\lambda(x,.)\|_{2}^2 \leq C_0 \frac 1 {|z+\frac 1 2| (\Re(z))},
 \end{equation}
where $C_0$ is as defined in (\ref{eq:5}).
\end{cor}
\begin{proof}[Proof of the corollary] 
Let us first consider the case $\lambda\in \mc \setminus \mr, \lambda=-s(s-1)$. Since 
$$ \psi(s) = -\gamma + \sum_{k=1}^\infty \left( \frac 1 k - \frac 1 {k+s-1}\right),$$
where $\gamma$ is the Euler-Mascheroni constant (see \cite[Formula 6.3.16]{MR0167642}), we can use the fact that $\Re(s) > 1/2$ to obtain that  
\begin{eqnarray*}
&& |\Im(\psi(s))| = |\Im(s)| \sum_{k=0}^\infty \frac 1 {(k+\Re(s))^2 + \Im(s)^2} 
\leq |\Im(s)| \sum_{k=0}^\infty \frac 1 {k^2+|s|^2} \\
&=& \frac{|\Im(s)|} 2 \left(\frac 1 {|s|^2} + \frac{\pi \coth(\pi |s|)}{|s|}\right) 
\leq \frac{|\Im(s)|}{2|s|}  \left(2 + \pi \coth(\frac \pi 2 )\right).
\end{eqnarray*}
Moreover, a short computation shows that 
$$ \Im(\lambda)= 2 \Im(s) (\frac 1 2-\Re(s)).$$
Hence for $\lambda \in \mc \setminus \mr$ we obtain that
\begin{equation}
\frac{|\Im(\psi(s))|} {2 \pi |\Im(\lambda)|} \leq \frac{1 + \frac \pi 2 \coth(\frac \pi 2 )}{4\pi} \frac 1 {|s|  (\Re(s)-\frac 1 2)}.\label{eq:22}
\end{equation}
Similarly, for $\lambda < 1/ 4$ (and hence $s>1/2$) we use that
\begin{eqnarray*}
\psi'(s) = \sum_{k=0}^\infty \frac 1 {(s+k)^2} \leq \sum_{k=0}^\infty \frac 1 {s^2+k^2} \leq \frac 1 s \left(1 + \frac \pi 2 \coth(\frac \pi 2)\right)
\end{eqnarray*}
to  obtain
\begin{equation}
 \frac{\psi'(s)} {4 \pi (s-\frac 1 2)} \leq  \frac{1 + \frac \pi 2 \coth(\frac \pi 2 )}{4\pi} \frac 1 {s(s-\frac 1 2)}, \qquad \lambda < 1/4. \label{eq:23} 
 \end{equation}
Taking into account that by (\ref{eq:15}) we have $s=z+\frac 1 2$, the estimates (\ref{eq:22}), (\ref{eq:23}) and (\ref{eq:20}) conclude the proof.
\end{proof}
Now we estimate the $L_1$-norm of the Green function.
\begin{lemma}\label{lemma1}
 For all $z \in \mc_+$ and $\lambda=\Psi_1(z)=\frac 1 4 -(z+\frac 1 2)^2$ we have
 \begin{equation}
     \label{eq:24}
   \sup_{x \in \mh} \|G_\lambda(x,.)\|_{1} \leq \frac 1 {\Re(z)(\Re(z)+1)}.
\end{equation}
\end{lemma}
\begin{proof}
We note that with $d=d(x,x')$ we obtain from (\ref{eq:9}) and (\ref{eq:15}) that
$$  G_\lambda(x,x') = \frac{1}{2^{\frac 3 2} \pi} \int_d^\infty \frac{e^{-a(z+\frac 1 2)}}{(\cosh(a)-\cosh(d))^{1/2}} da.$$
Switching to geodesic polar coordinates, centered at $x$,  we can thus compute 
  \begin{eqnarray*}
&&    \|G_\lambda(x,.)\|_{1} = \int_\mh \mu(dx') |G_\lambda(x,x')| \\
&\leq& \frac{1}{2^{\frac 3 2} \pi} \int_{\ms^1} d\xi \int_0^\infty dr \sinh(r) \int_r^\infty da \frac{e^{-a(\Re(z)+\frac 1 2)}}{(\cosh(a)-\cosh(r))^{1/2}} \\
&=& \frac{1}{2^{\frac 1 2}} \int_0^\infty da\: e^{-a(\Re(z)+\frac 1 2)}\int_0^a dr  \frac{\sinh(r) }{(\cosh(a)-\cosh(r))^{1/2}} \\
&=& 2^{\frac 1 2} \int_0^\infty da\: e^{-a(\Re(z)+\frac 1 2)} (\cosh(a)-1)^{1/2}
= 
    \frac 1 {\Re(z)(\Re(z)+1)}.
  \end{eqnarray*}
Here in the last step the integral has been evaluated using Mathematica. 
\end{proof}
Finally, we generalize the previous two lemmas using complex interpolation, see Appendix \ref{ap:inter}. 
\begin{lemma} 
Let $1 \leq p < 2$. Then for $z \in \mc_+$ and $\lambda=\Psi_p(z)=\frac 1 4 -(z+\gamma_p)^2$
we have
\begin{equation}
  \label{eq:25}
   \sup_{x \in \mh} \|G_\lambda(x,.)\|_{p} \leq C_0^{1-1/p} \left(\frac 1 {\Re(z)(\Re(z)+\frac 1 2)}\right)^{1/p},
\end{equation}
where $C_0$ was defined in (\ref{eq:5}).
\end{lemma}

\begin{proof}
We use the terminology of Appendix \ref{ap:inter}. Let $S:=\{ w : 0 \leq \Re(w) \leq 1\}$. For fixed $x \in \mh$ and $z \in \mc_+$ we consider the function 
\begin{eqnarray*}
 f : S \to L_2(\mh)+L_1(\mh), \quad w \mapsto G_{\frac 1 4 - (z + \frac 1 2 w)^2}(x,.)
\end{eqnarray*}
The explicit expression (\ref{eq:9}) for the kernel and our above estimates show that this function is in $\mathcal G(L_2(\mh),L_1(\mh))$, i.e. it is continuous and bounded on $S$ and analytic in the interior of $S$. Moreover, by (\ref{eq:21})
\begin{eqnarray*}
A_0^2 := \sup_{y \in \mr} \|f(iy)\|_{2}^2 \leq C_0 \sup_{y \in \mr}\frac 1 {|z+i\frac y 2 +\frac 1 2| (\Re(z+i \frac y 2))} = C_0 \frac 1 {(\Re(z)+\frac 1 2)(\Re(z))}
\end{eqnarray*}
and by (\ref{eq:24})
\begin{eqnarray*}
  A_1 := \sup_{y \in \mr} \|f(1+iy)\|_{1} \leq \sup_{y \in \mr} \frac 1 {\Re(z+i \frac y 2)(\Re(z+ i \frac y 2)+1)} = \frac 1 {\Re(z)(\Re(z)+1)}.
\end{eqnarray*}  
Hence from Proposition \ref{prop:inter} we obtain that for $\theta \in (0,1)$ and $1/p= (1-\theta)/2+\theta$ we have $f(\theta) \in L_p(\mh)=[L_2(\mh),L_1(\mh)]_\theta$ and 
\begin{eqnarray*}
  \|f(\theta)\|_{p} =\|f(\theta)\|_{[L_2,L_1]_\theta} \leq A_0^{1-\theta}A_1^\theta.
\end{eqnarray*}
But using that  $\theta = 2/p-1, \gamma_p=1/p-1/2$ and $\lambda=\frac 1 4 - (z+\gamma_p)^2=\frac 1 4 - (z+\frac 1 2 \theta)^2$ the last bound translates into
{\small
\begin{eqnarray*} 
  \|G_\lambda(x,.)\|_{p} &\leq& C_0^{(1-1/p)} \left(\frac 1 {(\Re(z)+\frac 1 2)(\Re(z))}\right)^{(1-1/p)} \left(\frac 1 {\Re(z)(\Re(z)+1)}\right)^{2/p-1} \\
&\leq&  C_0^{(1-1/p)} \left(\frac 1 {\Re(z)(\Re(z)+\frac 1 2)}\right)^{1/p}, 
\end{eqnarray*}
}
where in the last step we used that $\Re(z)+1 > \Re(z)+1/2$.
\end{proof}

\subsection{A resolvent norm estimate}

We continue with an estimate on the operator norm of the resolvents of $-\Delta_p$. Here and in the following we write $\|T\|_{p,q}$ for the operator norm of $T:L_p(\mh) \to L_q(\mh)$.
\begin{lemma}\label{lem:2}
Let $1 \leq p < \infty$ and  let $z \in \mc_+$ and $\lambda= \Psi_p(z)= \frac 1 4 - (z+\gamma_p)^2$. 
\begin{enumerate}
\item[(i)] If $p=2$, then
$$ \|(-\Delta-\lambda)^{-1}\|_{2,2} = \frac{1}{\dist(\lambda,[1/4,\infty))} \leq \frac 1 {|z| \Re(z)}.$$
\item[(ii)] If $p \neq 2$, then
$$     \|(-\Delta_p-\lambda)^{-1}\|_{p,p} \leq \left( \frac 1 {(\Re(z))^{2-2\gamma_p}(1+\Re(z))^{2\gamma_p}} \right). $$
\end{enumerate}
\end{lemma} 

\begin{proof}
(i) The identity follows from the fact that $-\Delta=-\Delta_2$ is selfadjoint with $\sigma(-\Delta_2)=[1/4,\infty)$. The inequality follows from Lemma \ref{lem:41} (i).

(ii) In case $p=1$ we can use Lemma \ref{lemma1} to compute for $\lambda = \frac 1 4 -
(z+\frac 1 2)^2$ 
\begin{equation}
  \label{eq:26}
  \|(-\Delta_1-\lambda)^{-1}\|_{1,1} \leq \sup_{x \in \mh} \|G_\lambda(.,x)\|_{1} = \sup_{x \in \mh} \|G_\lambda(x,.)\|_{1} \leq \frac 1 {\Re(z)(\Re(z)+1)}.
\end{equation}
Now we treat the case $1<p<2$ by interpolation (see again Appendix \ref{ap:inter}): Let $S=\{ w \in \mc : 0 \leq \Re(w) \leq 1\}$ and fix $z \in \mc_+$. Define $$T_w:= (-\Delta-\frac 1 4 + (z+\frac 1 2 w)^2)^{-1}.$$ 
Then for all simple functions $f,g: \mh \to \mc$ the function   
\begin{eqnarray*} 
S \ni  w &\mapsto& \int_\mh T_wf(x) g(x) \mu(dx)
\end{eqnarray*}
is continuous and bounded on $S$ and analytic in the interior of $S$. Moreover, for every simple function $f$ we have
\begin{eqnarray*}
 && \sup_{y \in \mr} \|T_{iy}f\|_{2} \leq \|f\|_2  \sup_{y \in \mr} \|(-\Delta-\frac 1 4 + (z+\frac 1 2 (iy))^2)^{-1}\|_{2,2} \\
  &\leq& \|f\|_2 \sup_{y \in \mr} \frac 1 {|z+i\frac{y} 2|\Re(z+i \frac{y} 2)} = \|f\|_2 \frac 1 {(\Re(z))^2}
\end{eqnarray*}
and
\begin{eqnarray*}
&&  \sup_{y \in \mr} \|T_{1+iy}f\|_{1} \leq \|f\|_1   \sup_{y \in \mr} \|(-\Delta_1-\frac 1 4 + (z+\frac 1 2(1+iy))^2)^{-1}\|_{1,1} \\
&\leq& \|f\|_1 \sup_{y \in \mr} \frac 1 {\Re(z+i\frac y 2)(\Re(z+i\frac y 2)+1)} = \|f\|_1 \frac 1 {\Re(z)(1+\Re(z))}.
\end{eqnarray*}
Hence the Stein interpolation theorem (Theorem \ref{thm:stein}) implies that for $\theta \in (0,1)$ and $\frac 1 p = (1-\theta) \frac 1 2 + \theta$, the operator $T_\theta$ extends to a bounded operator on $L_p(\mh)$ satisfying 
$$ \|T_\theta\|_{p,p} \leq \left(\frac 1 {(\Re(z))^2}\right)^{1-\theta}\left(\frac 1 {\Re(z)(1+\Re(z))}\right)^\theta.$$
Since $1<p<2$ we have $\theta = (2/p-1)=2\gamma_p$ (see Definition (\ref{eq:7})) and 
$\lambda=\frac 1 4 -(z+\gamma_p)^2=\frac 1 4 -(z+\frac 1 2 \theta)^2$. Hence the previous estimate implies
$$ \|(-\Delta_p-\lambda)^{-1} \|_{p,p} \leq \frac 1 {(\Re(z))^{2-2\gamma_p}(1+\Re(z))^{2\gamma_p}}.$$
Finally, the case $p>2$ follows by duality using the fact that $\gamma_p= \gamma_{p'}$.
\end{proof}

\subsection{Summing norm estimates}

In this section $(\Pi_r, \|.\|_{\Pi_r})$ and $(\Pi_{r,q},\|.\|_{\Pi_{r,q}})$ denote the $r$-summing and $(r,q)$-summing operators on $L_p(\mh)$, respectively. Some properties of these operator ideals are reviewed in Appendix \ref{ap:ideals} (see Examples \ref{ex:schatt} and \ref{ex:summ}, in particular).
\begin{lemma}\label{lem:4}
Let $p \geq 2$, $z \in \mc_+$ and $\lambda=\Psi_p(z)$. If $V \in L_\infty(\mh)$, then
\begin{equation}
\|  V(-\Delta_p-\lambda)^{-1}\|_{p,p} \leq \|V\|_\infty \cdot \left\{
  \begin{array}{cl}
    \frac 1 {\dist(\lambda,[1/4,\infty))}, & p=2, \\[6pt]
    \frac 1 {(\Re(z))^{1+\frac 2 p}(1+\Re(z))^{1-\frac 2 p}}, & p > 2.
  \end{array}\right.
\label{eq:27}
\end{equation}
\end{lemma}
\begin{proof}
 This follows from Lemma \ref{lem:2} and the fact that 
$$ \|  V(-\Delta_p-\lambda)^{-1}\|_{p,p} \leq \|  V\|_\infty \|(-\Delta_p-\lambda)^{-1}\|_{p,p}.$$
\end{proof} 

\begin{lemma}\label{lem:5}
  Let $p \geq 2$, $z \in \mc_+$ and $\lambda=\Psi_p(z)$. If $V \in L_p(\mh)$, then $V(-\Delta_p-\lambda)^{-1} \in \Pi_p(L_p(\mh))$ and
  \begin{equation}
    \label{eq:28}
    \|  V(-\Delta_p-\lambda)^{-1}\|_{\Pi_p}^p \leq C_0 \cdot \|V\|_p^p \cdot \left\{
  \begin{array}{cl}
 \frac 1 {|z+\frac 1 2| (\Re(z))}, & p =2 \\[6pt]
 \left( \frac 1 {\Re(z)(\Re(z) + \frac 1 2)} \right)^{p/p'}, & p > 2,
  \end{array}\right.
  \end{equation}
where $C_0$ was defined in (\ref{eq:5}).
\end{lemma}
\begin{proof}
Since $V(-\Delta_p-\lambda)^{-1}$ is an integral op. with kernel $k(x,x')=V(x)G_\lambda(x,x')$ its $p$-summing norm can be computed as follows (see, e.g., \cite[Thm.3.a.3 and its proof]{MR889455}): 
\begin{eqnarray*}
&& \|V(-\Delta_p-\lambda)^{-1}\|_{\Pi_p}^p \leq \int_\mh \mu(dx) \left(\int_\mh \mu(dx') |V(x)G_\lambda(x,x')|^{p'} \right)^{p/p'}  \\
&\leq& \|V\|_p^p \sup_{x \in \mh} \|G_\lambda(x,.)\|_{L_{p'}}^p \\
&\leq& C_0 \|V\|_p^p \cdot \left\{
  \begin{array}{cl}
 \frac 1 {|z+\frac 1 2| (\Re(z))}, & p =2 \\[6pt]
 \left( \frac 1 {\Re(z)(\Re(z) +\frac 1 2)} \right)^{p/p'}, & p > 2.
  \end{array}\right.
\end{eqnarray*}
Here in the last inequality we used (\ref{eq:21}) and (\ref{eq:25}), respectively.
\end{proof} 
Now we are going to interpolate between the results of the last two lemmas to obtain a result for $V \in L_r(\mh), p < r < \infty.$ We will need the following result of Pietsch and Triebel concerning the complex interpolation spaces of the Schatten-von Neumann and absolutely summing ideals, respectively. We refer again to Appendix \ref{ap:inter} for the notation and terminology. 
\begin{prop}[\cite{MR0243341}]\label{prop:triebel}
Let $\hil$ and $X$ denote complex Hilbert and Banach spaces, respectively. Moreover, let $p \geq 1, 0 < \theta < 1$ and define $r>p$ by $ \frac 1 r = \frac {\theta} p$.
Then the following holds:
\begin{enumerate}
    \item[(i)] $[\bdd(\hil),\mathcal S_p(\hil)]_\theta = \mathcal S_{r}(\hil)$ and $\|T\|_{[\bdd,\mathcal S_p]_\theta}=\|T\|_{\mathcal S_r}$ for $T \in \mathcal S_r(\hil)$.\\
\item[(ii)] $ [\bdd(X), \Pi_p(X)]_\theta \subset \Pi_{r,p}(X)$  and  
$$ \|T\|_{\Pi_{r,p}} \leq \|T\|_{[\bdd,\Pi_p]_\theta} \text{ for } T \in [\bdd(X),\Pi_p(X)]_\theta.$$ 
\end{enumerate}
\end{prop}
\begin{rem}
  We recall that for $p=2$ and $r \geq 2$ we have
  $\Pi_{r,2}(\hil)=S_r(\hil)$ and $\|.\|_{\Pi_{r,2}}=\|.\|_{\mathcal S_r}$, see \cite[Prop. 2.11.28]{MR917067}.
\end{rem} 

\begin{thm}\label{thm:6}
  Let $2 \leq p \leq r < \infty$, $z \in \mc_+$ and $\lambda=\Psi_p(z)=1/4-(z+\gamma_p)^2$. If $V \in L_r(\mh)$, then $V(-\Delta_p-\lambda)^{-1} \in \Pi_{r,p}(L_p(\mh))$ and
{\small
  \begin{eqnarray*}
\|  V(-\Delta_p-\lambda)^{-1}\|_{\Pi_{r,p}} 
&\leq& C_0^{1/r} \|V\|_r \cdot \left\{ 
  \begin{array}{cl}
\left(\frac 1 {\dist(\lambda,[1/4,\infty))}\right)^{1-2/r} \left(\frac 1 {|z+\frac 1 2| (\Re(z))} \right)^{1/r}, & p=2, \\[6pt]
    \left(\frac 1 {\Re(z)}\right)^{1+2/p-3/r} \left( \frac 1 {\Re(z) + \frac 1 2} \right)^{1-2/p+1/r}, & p > 2,
  \end{array}\right.
  \end{eqnarray*}
} 
where $C_0$ was defined in (\ref{eq:5}). 
\end{thm}
In particular, the theorem shows that for $V \in L_r(\mh)$ the operator of multiplication by $V$ is $-\Delta_p$-compact. This was used in Section \ref{sec:32}. 
\begin{proof}
A density argument shows that it is sufficient to consider the case where $V$ is a nonnegative simple function. For such a $V$ define 
$$f: S \to \bdd(L_p(\mh)) + \Pi_p(L_p(\mh)), \qquad f(w)=V^{\frac r p w}(-\Delta_p-\lambda)^{-1},$$
where as above $S = \{ w : 0 \leq \Re(w) \leq 1\}$. From what we have shown above we infer that $f$ is continuous and bounded on $S$ and holomorphic in the interior of $S$. Moreover, since $\|V^{\frac r p iy}\|_{\infty} \leq 1$ we obtain from Lemma \ref{lem:4} that
\begin{eqnarray*}
  A_0:= \sup_{y \in\mr} \|f(iy)\|_{p,p} \leq \left\{
  \begin{array}{cl}
    \frac 1 {\dist(\lambda,[1/4,\infty))}, & p=2, \\[6pt]
    \frac 1 {(\Re(z))^{1+\frac 2 p}(1+\Re(z))^{1-\frac 2 p}}, & p > 2.
  \end{array}\right.      
\end{eqnarray*}
Furthermore, Lemma \ref{lem:5} implies that
\begin{eqnarray*}
  A_1^p &:=& \sup_{y \in \mr} \|f(1+iy)\|_{\Pi_p(L_p)}^p \leq C_0 \cdot \|V^{\frac r p}\|_p^p \cdot \left\{
  \begin{array}{cl}
 \frac 1 {|z+\frac 1 2| (\Re(z))}, & p =2 \\[6pt]
 \left( \frac 1 {\Re(z)(\Re(z) +\frac 1 2)} \right)^{p/p'}, & p > 2,
  \end{array}\right.
\end{eqnarray*}
and here $\|V^{\frac r p}\|_p^p = \|V\|_r^r$. But then Proposition \ref{prop:inter} and Proposition \ref{prop:triebel} imply that with $1/r= \theta / p$ (i.e. $f(\theta)= V(-\Delta_p-\lambda)^{-1}$)
{\small
\begin{eqnarray*}
&&  \|  V(-\Delta_p-\lambda)^{-1}\|_{\Pi_{r,p}(L_p(\mh))} 
\leq \|  V(-\Delta_p-\lambda)^{-1}\|_{[\bdd,\Pi_p]_\theta} 
\leq A_0^{(1-\theta)}A_1^{\theta} \\
&\leq& C_0^{1/r} \|V\|_r\left\{
  \begin{array}{cl}
    \left(\frac 1 {\dist(\lambda,[1/4,\infty))}\right)^{1-2/r} \left(\frac 1 {|z+\frac 1 2| (\Re(z))} \right)^{1/r}, & p=2, \\[6pt]
    \left(\frac 1 {(\Re(z))^{1+\frac 2 p}(1+\Re(z))^{1-\frac 2 p}}\right)^{1-p/r} \left( \frac 1 {\Re(z)(\Re(z) +\frac 1 2)} \right)^{p/(rp')}, & p > 2. 
  \end{array}\right. 
\end{eqnarray*} 
}
Now a rearrangement of terms, using the estimate $\Re(z)+1 > \Re(z) + \frac 1 2$, concludes the proof.
\end{proof}
The previous theorem will be used to prove the results in Section \ref{sec:21}. To prove the results in Section \ref{sec:22} we will use the following corollary. 
\begin{cor}\label{cor:est} 
  Let $2 \leq p \leq r < \infty$, $z \in \mc_+$ and $\lambda=\Psi_p(z)$. If $V \in L_r(\mh)$, then $V(-\Delta_p-\lambda)^{-1} \in \Pi_{r,p}(L_p(\mh))$ and
{\small
  \begin{eqnarray*}
\|  V(-\Delta_p-\lambda)^{-1}\|_{\Pi_{r,p}} 
&\leq&  2^{1-2/p} \|V\|_r \left\{ 
  \begin{array}{cl}
    \left(\frac 1 {|z|}\right)^{1-2/r}  \left( \frac 1 {\Re(z)} \right)^{1-\frac 1 r}, & p=2, \\[6pt]
    \left(\frac 1 {\Re(z)}\right)^{1+2/p-3/r}, & p > 2.
  \end{array}\right.
  \end{eqnarray*}
}
\end{cor}
\begin{proof}
The case $p=2$ follows using Lemma \ref{lem:41} to estimate $\dist(\lambda,[1/4,\infty)) \geq |z|\Re(z)$, together with the estimate $|z+1/2| \geq 1/2$ for $z \in \mc_+,$ and the fact that $2C_0 \leq 1$. The case $p>2$ follows in the same way using that  $\Re(z)+\frac 1 2 \geq \frac 1 2$ and $C_0^{1/r}2^{1-2/p+1/r} \leq 2^{1-2/p}$.
\end{proof}

\section{Proof of Theorem \ref{thm:1} and \ref{thm:2}} \label{sec:5}

Let $p \geq 2$ and let $\lambda \in \Sigma_p^c=\rho(-\Delta_p)$. Then by the Birman-Schwinger principle $\lambda$ is an eigenvalue of $H_p=-\Delta_p+V$ if and only if $-1$ is an eigenvalue of $V(-\Delta_p-\lambda)^{-1}$. Hence in this case we obtain  for $r \geq p$ that
\begin{equation}
1 \leq \|V(-\Delta_p-\lambda)^{-1}\| \leq \|  V(-\Delta_p-\lambda)^{-1}\|_{\Pi_{r,p}}.\label{eq:29}
\end{equation}
Now we use Theorem \ref{thm:6} to estimate the right-hand side from above and we rearrange terms. We distinguish between two cases:

(i) In case $p=2$ we obtain with $z=\Psi_2^{-1}(\lambda)=\sqrt{1/4-\lambda}$
\begin{equation}
(\dist(\lambda,[1/4,\infty)))^{r-2}|z+1/2|\Re(z) \leq C_0 \|V\|_r^r.\label{eq:30}
\end{equation}
A short calculation shows that, since $ z \in \mc_+$, we have $|z+1/2| \geq 1/\sqrt{2} \cdot (|z|+1/2)$. Hence, using Lemma \ref{lem:41} we see that the  left-hand side of (\ref{eq:30}) can be bounded from below as follows:
\begin{eqnarray}
&& (\dist(\lambda,[1/4,\infty)))^{r-2}|z+1/2|\Re(z) \nonumber \\
&\geq& 1/\sqrt{2}\cdot (\dist(\lambda,[1/4,\infty)))^{r-2}(|z|+1/2)\Re(z) \nonumber \\
&=& 1/\sqrt{2}\cdot (\dist(\lambda,[1/4,\infty)))^{r-2} (1+1/(2|z|)) |z| \Re(z) \nonumber \\
&\geq& 1/(2\sqrt{2}) \cdot (\dist(\lambda,[1/4,\infty)))^{r-1} |1+1/(2|1/4-\lambda|^{1/2})| \label{eq:x}.
\end{eqnarray}  
But (\ref{eq:30}) and (\ref{eq:x}) show the validity of (\ref{eq:4}) and conclude the proof of Theorem \ref{thm:1}

(ii) In case $p > 2$ we obtain from (\ref{eq:29}) and Theorem \ref{thm:6} that
\begin{equation}
  \label{eq:31}
  (\Re(z))^{2r-2}(1+1/(2\Re(z)))^{r(1-2/p)+1} \leq C_0 \|V\|_r^r.
\end{equation}
Now we use Lemma \ref{lem:41} and (\ref{eq:14}) to estimate the left-hand side from below by
$$\left( \frac 1 {16} \frac{\dist(\lambda,\Sigma_p)}{|1/4-\lambda|^{1/2}} \right)^{2r-2} \left( 1 + \frac{|1/4-\lambda|^{1/2}}{8 \dist(\lambda,\Sigma_p)}\right)^{r(1-2/p)+1}.$$
This shows the validity of (\ref{eq:6}) in case $p > 2$. Finally, the case $1<p<2$ follows by duality using Proposition \ref{prop:1}. This concludes the proof of Theorem \ref{thm:2}.

\section{An abstract Lieb-Thirring estimate} \label{sec:6}

The Theorems \ref{thm:3} and \ref{thm:4} will be proved using the following abstract result. Here we use terminology from \cite{Hansmann15}, which is reviewed in Appendix \ref{ap:ideals}. Moreover, $x_+=\max(x,0)$ denotes the positive part of $x \in \mr$.

\begin{thm}\label{thm:61}
Let $X$ denote a complex Banach space, let $r \geq 1$  and let $(\mathcal I,\|.\|_{\mathcal I})$ be an $l_r$-ideal in $\bdd(X)$. Moreover, let $Z_0$ and $Z=Z_0+M$ denote closed operators in $X$ such that
\begin{itemize}
\item[-] for some $p \in [1,\infty)$ we have $\sigma(Z_0)=\Sigma_p$ as defined in (\ref{eq:1}),
\item[-] there exist $\alpha, \beta,\gamma \geq 0$ and $C_1> 0$ such that for all $z \in \mc_+$ and $\Psi_p(z)$ as defined in (\ref{eq:13}):
\begin{equation}
    \label{eq:32}
  \|M(Z_0-\Psi_p(z))^{-1}\|_{\mathcal{I}} \leq C_1\cdot\Re(z)^{-\alpha}  \cdot |z|^{-\beta}
\end{equation}
and 
\begin{equation}
    \label{eq:33}
  \|(Z_0-\Psi_p(a))^{-1}\| \leq a^{-\gamma}, \qquad a>0.
\end{equation}
\end{itemize}
Finally, let $\tau > 0$ and set  
\begin{eqnarray*}
  \delta_1&:=& r\alpha + 1+ \tau, \\
  \delta_2 &:=& (r\beta-1+\tau)_+, \\
  \delta_3 &:=& r(1-\alpha-\beta-\gamma).
\end{eqnarray*} 
Then there exist constants $C$ and $C'$, both depending on $\alpha,\beta,\gamma,r$ and $\tau$, such that 
{\small
  \begin{equation}
   \sum_{|\lambda- \frac 1 {pp'}|^{\frac 1 2} \leq (2C_1)^{\frac 1 {\alpha+\beta}}}  \frac{\dist(\lambda,\Sigma_p)^{\delta_1} |\lambda-\frac 1 {pp'}|^{\delta_2}}{(|\lambda- \frac 1 {pp'}|^{\frac 1 2}+2\gamma_p)^{\delta_1+\delta_2}}
\leq C \cdot C_1^{r+\frac{\delta_1+\delta_2+\delta_3}{\alpha+\beta}}(C_1^{\frac 1 {\alpha + \beta}}+\gamma_p)^{r}\label{eq:34}
\end{equation}
}
and
{\small
  \begin{equation}
 \sum_{|\lambda- \frac 1 {pp'}|^{\frac 1 2}>(2C_1)^{\frac 1 {\alpha+\beta}}}  \frac{\dist(\lambda,\Sigma_p)^{\delta_1} |\lambda-\frac 1 {pp'}|^{\delta_2}}{(|\lambda-\frac 1 {pp'}|^{\frac 1 2}+2\gamma_p)^{2\delta_1+2\delta_2 + \delta_3+r+\tau}} 
\leq C' \cdot C_1^{r}(C_1^{\frac 1 {\alpha+\beta}}+\gamma_p)^{-\tau}
.\label{eq:35}
\end{equation} 
}
Here in both sums we are summing over all eigenvalues $\lambda \in \sigma_d(Z)\cap \Sigma_p^c$ satisfying the stated restrictions, each eigenvalue being counted according to its algebraic multiplicity. Moreover, $\gamma_p$ is as defined in (\ref{eq:7}).
\end{thm}
In the remainder of this section we are going to prove the previous theorem. We start with a lemma providing a resolvent norm estimate on $Z=Z_0+M$.

\begin{lemma}\label{lem:62} 
Given assumptions (\ref{eq:32}) and (\ref{eq:33}) we have
for all  
\begin{equation}
 a \geq (2C_1)^{1/(\alpha + \beta)}\label{eq:36}
\end{equation}
that $\Psi_p(a) \in \rho(Z)$ and 
\begin{equation}
  \label{eq:37}
    \|(Z-\Psi_p(a))^{-1}\| \leq 2 a^{-\gamma}.
\end{equation}
\end{lemma}
\begin{proof}
Since for $a>0$ we have $\Psi_p(a) \in \Sigma_p^c=\rho(Z_0)$ we can write
\begin{equation}
Z-\Psi_p(a) = Z_0+M-\Psi_p(a) = (I+M(Z_0-\Psi_p(a))^{-1})(Z_0-\Psi_p(a)).\label{eq:38}
\end{equation}
By assumption (\ref{eq:32})
\begin{eqnarray*}
  \|M(Z_0-\Psi_p(a))^{-1}\| \leq   \|M(Z_0-\Psi_p(a))^{-1}\|_{\mathcal I}
\leq C_1 a^{-\alpha-\beta}.
\end{eqnarray*}
Hence we see that for
$a \geq (2C_1)^{1/(\alpha + \beta)}$ the operator $I+M(Z_0-\Psi_p(a))^{-1}$ is invertible with norm of the inverse being at most $2$. But then also $Z-\Psi_p(a)$ is invertible and using (\ref{eq:33}) and (\ref{eq:38}) we obtain
\begin{eqnarray*}
  \|(Z-\Psi_p(a))^{-1}\| \leq 2 \|(Z_0-\Psi_p(a))^{-1}\| \leq 2 a^{-\gamma}.
\end{eqnarray*}  
\end{proof}
Now for a shorter notation let us set 
$$ b= \Psi_p(a)$$   
with some $a$ satisfying (\ref{eq:36}). Then $b \in \rho(Z) \cap \rho(Z_0)$ and 
$$ K:=(Z-b)^{-1} - (Z_0-b)^{-1} = -(Z-b)^{-1}M(Z_0-b)^{-1} \in \mathcal I.$$
Let us set $A=(Z_0-b)^{-1}$ and $B:=A+K:=(Z-b)^{-1}$. By the spectral mapping theorem
$$ \lambda \in \rho(Z_0)\setminus \{b\} \quad \Leftrightarrow \quad (\lambda-b)^{-1} \in \rho(A)$$
(and a similar result holds for $Z$ and $B$). From \cite[Theorem 4.10]{Hansmann15}, see Appendix \ref{ap:ideals}, we know that there exists a holomorphic function
$d : \rho(A) \to \mc$ with the following properties:
\begin{enumerate}
    \item[(p1)] $\lim_{|u| \to \infty} d(u)=1$,
    \item[(p2)] for $u \in \rho(A)$ we have
  \begin{equation*}
    |d(u)| \leq \exp\left( \mu_r^r \Gamma_r \|K(u-A)^{-1}\|_{\mathcal I}^r\right),
  \end{equation*}
 where $\mu_r$ denotes the eigenvalue constant of $\mathcal I$ and $\Gamma_r$ is a universal $r$-dependent constant, see \cite{Han17a},
\item[(p3)] $d(u)=0$ iff $u \in \sigma(A+K)$,
\item[(p4)] if $u \in \rho(A) \cap \sigma_d(A+K)$, then its algebraic multiplicity (as an eigenvalue) coincides with its order as a zero of $d$.
\end{enumerate}
Using the spectral mapping theorem again we see that
\begin{equation*} 
  D(\lambda):=d((\lambda-b)^{-1})
\end{equation*}
is well-defined and analytic on $\rho(Z_0) \setminus \{b\}$ and, by (p1), can be analytically extended to $\rho(Z_0)=\Sigma_p^c$ by setting $D(b)=1$. Moreover, by spectral mapping and (p3) and (p4) we know that $D(\lambda)=0$ iff $\lambda \in \sigma(Z)$ and if $\lambda \in \rho(Z_0) \cap \sigma_d(Z)$, then its algebraic multiplicity coincides with its order as a zero of $D$. Finally, since
$$((\lambda-b)^{-1}-A)^{-1}=((\lambda-b)^{-1}-(Z_0-b)^{-1})^{-1}=(\lambda-b)(Z_0-b)(Z_0-\lambda)^{-1}$$
we see that
\begin{eqnarray*}
 K((\lambda-b)^{-1}-A)^{-1} &=& ((Z-b)^{-1}-(Z_0-b)^{-1})(\lambda-b)(Z_0-b)(Z_0-\lambda)^{-1} \\
&=&(b-\lambda)(Z-b)^{-1}M(Z_0-\lambda)^{-1}
\end{eqnarray*}
and hence for $\lambda \in \rho(Z_0)=\Sigma_p^c$ we have by (p2)
\begin{eqnarray*}
    |D(\lambda)| &\leq& \exp\left( \mu_r^r \Gamma_r |\lambda-b|^r \|(Z-b)^{-1}\|^r \|M(Z_0-\lambda)^{-1}\|_{\mathcal I}^r\right).
 \end{eqnarray*}
Writing $b=\Psi_p(a)$ and $\lambda=\Psi_p(z)$, with $z \in \mc_+$, the assumption (\ref{eq:32}) and Lemma \ref{lem:62} hence imply that
\begin{equation}
|D(\Psi_p(z))| \leq \exp\left( \mu_r^r \Gamma_r |\Psi_p(z)-\Psi_p(a)|^r 2^ra^{-r\gamma} C_1^r \Re(z)^{-r\alpha} |z|^{-r\beta}\right).   \label{eq:41}  
\end{equation}
Here the holomorphic function $D \circ \Psi_p$ is defined on the right half-plane $\mc_+$. In the following, it will be necessary to transfer this function to the unit disk $\md$ using the conformal map  
$$ \Phi_a: \md \to \mc_+, \quad \Phi_a(w) = a \frac{1-w}{1+w}$$
with inverse 
$$ \Phi_a^{-1}(z)=\frac{a-z}{a+z}.$$
\begin{lemma}\label{lem63}
  Let $w \in \md, z=\Phi_a(w) \in \mc_+$ and $\lambda= \Psi_p(z) \in \Sigma_p^c$. Then the following holds:
  \begin{align}
   1+w = \frac{2a}{a+z}, &\quad 1-w=\frac{2z}{a+z}     \label{eq:13a}\\
    |\Psi_p(z)-\Psi_p(a)| &\leq \frac{4a(a+2\gamma_p)}{|1+w|^2}     \label{eq:13b}\\
    a \frac{1-|w|}{|1+w|^2} & \leq \Re(z)  \leq 2a \frac{1-|w|}{|1+w|^2}     \label{eq:13c}\\
   \frac{a\cdot \dist(\lambda,\Sigma_p)}{8|a+z|^2|1/4-\lambda|^{1/2}}  &\leq 1-|w| \leq \frac{16a\cdot\dist(\lambda,\Sigma_p)}{|a+z|^2|1/4-\lambda|^{1/2}}. \label{eq:42}
\end{align}
\end{lemma}  
\begin{proof}[Proof of the lemma]
The identities in (\ref{eq:13a}) are immediate consequences of the definitions. To see (\ref{eq:13b}) we compute, using (\ref{eq:13}),
\begin{eqnarray*}
|\Psi_p(z)-\Psi_p(a)| &=& | (a+\gamma_p)^2-(z+\gamma_p)^2| = |a^2-z^2 + 2\gamma_p(a-z)|.
\end{eqnarray*}
Hence, since 
$$ a-z = a\left(1-\frac{1-w}{1+w}\right) = \frac{2aw}{1+w}, \qquad a+z = \frac{2a}{1+w},$$
we obtain
\begin{eqnarray*}
|\Psi_p(z)-\Psi_p(a)| &=&  \left| \frac{4wa^2}{(1+w)^2} +\frac{4wa\gamma_p}{1+w} \right| \\
 &=& \frac{4|w|a}{|1+w|^2} |a + \gamma_p(1+w)| \leq \frac{4a(a+2\gamma_p)}{|1+w|^2}.
\end{eqnarray*}  
The estimates in (\ref{eq:13c}) follow from
\begin{equation}
\Re(z)= a \frac{1-|w|^2}{|1+w|^2}.\label{eq:43}
\end{equation}
Finally, in order to show (\ref{eq:42}) we first use Lemma \ref{lem:41} to obtain
\begin{equation}
       \frac{|z+\gamma_p| \Re(z)}{4} \leq \dist(\lambda,\Sigma_p) \leq 16 |z+\gamma_p| \Re(z).\label{eq:44}
 \end{equation}
(here we ignore the fact that a better estimate is valid if $p=2$). Since  $z+\gamma_p=\sqrt{1/4-\lambda}$ we obtain, also using (\ref{eq:13a}) and (\ref{eq:43}), that
\begin{equation}
  |z+\gamma_p| \Re(z) = \frac{a |1/4-\lambda|^{1/2}(1-|w|^2)}{|1+w|^2} = (1-|w|^2) \frac{|1/4-\lambda|^{1/2}|a+z|^2}{4a}.\label{eq:21b}
\end{equation}
But (\ref{eq:21b}) and (\ref{eq:44}) imply  (\ref{eq:42}).
\end{proof}
Now let us introduce the holomorphic function 
$$h: \md \to \mc, \quad h(w)= D(\Psi_p(\Phi_a(w))).$$ Then $h(w)= 0$ if and only if $\lambda= \Psi_p(\Phi_a(w)) \in \sigma_d(Z) \cap \Sigma_p^c$ (and order and multiplicity coincide) and $h(0)=D(\Psi_p(a))=D(b)=1$. Moreover, using the previous lemma and (\ref{eq:41}) a short computation shows that 
{\small
\begin{eqnarray*}
|h(w)| &\leq&  \exp\left( 8^r\mu_r^r \Gamma_rC_1^r a^{r(1-\alpha-\beta-\gamma)}(a+2\gamma_p)^r \frac{1}{|1+w|^{r(2-2\alpha-\beta)}|1-w|^{r\beta}(1-|w|)^{r\alpha}}\right).
\end{eqnarray*}
} 
So we see that $h$ grows at most exponentially for $w$ approaching the unit circle, with the rate of explosion depending on whether $w$ approaches $1$ or $-1$ or a generic point of the boundary, respectively. A theorem of Borichev, Golinskii and Kupin \cite[Theorem 0.3]{MR2481997} allows us to transform this information on the growth of $h$ into the following information on its zero set: The theorem says that for all $\tau > 0$ there exists a constant $C=C(\alpha,\beta,\gamma,r, \tau) > 0$ such that
\begin{eqnarray}
&&  \sum_{h(w)=0, w \in \md} {(1-|w|)^{r\alpha+1+\tau}|1-w|^{(r\beta-1+\tau)_+}|1+w|^{(r(2-2\alpha-\beta) -1 +\tau)_+}} \nonumber\\ &\leq& C \cdot C_1^r \cdot a^{r(1-\alpha-\beta-\gamma)}(a+2\gamma_p)^r,   \label{eq:45}
\end{eqnarray}
where each zero of $h$ is counted according to its order. Using Lemma \ref{lem63} we see that the summands on the lhs are bounded below by
\begin{eqnarray*} 
\left( \frac a 8 \frac{\dist(\lambda,\Sigma_p)}{|a+z|^2|1/4 -\lambda|^{1/2}}\right)^{r\alpha+1+\tau} \left| \frac{2z}{a+z} \right|^{(r\beta-1+\tau)_+}\left| \frac{2a}{a+z}\right|^{(r(2-2\alpha-\beta)-1+\tau)_+}. 
 \end{eqnarray*}
Hence we have proved the following lemma. 
\begin{lemma}
Assume (\ref{eq:32}) and (\ref{eq:33}). 
Let $\tau > 0$ and set 
\begin{eqnarray*}
  \delta_1&:=& r\alpha + 1+ \tau, \\
  \delta_2 &:=& (r\beta-1+\tau)_+, \\
  \delta_3 &:=& r(1-\alpha-\beta-\gamma),\\
  \delta_4 &:=& (r(2-2\alpha-\beta)-1+\tau)_+. 
\end{eqnarray*}
Then there exists $C=C(\alpha,\beta,\gamma,r,\tau)$ such that for all $a \geq (2C_1)^{1/(\alpha + \beta)}$ we have 
{\small
  \begin{equation}
   \sum_{\lambda \in \sigma_d(Z)\cap \Sigma_p^c} \left( \frac{\dist(\lambda,\Sigma_p)}{|\frac 1 4 -\lambda|^{1/2}}\right)^{\delta_1} \frac{|z|^{\delta_2}}{|a+z|^{2\delta_1+\delta_2+\delta_4}} \leq C \cdot C_1^r \cdot a^{-\delta_1-\delta_4+\delta_3}(a+2\gamma_p)^r.\label{eq:46}
    \end{equation}
}
\noindent Here each eigenvalue is counted according to its alg. multiplicity and $z=\Psi_p^{-1}(\lambda)$.
\end{lemma} 
In order to finish the proof of Theorem \ref{thm:61} we need to distinguish between 'small' and 'large' eigenvalues. Namely, introducing 
$$0< \eta:= (2C_1)^{1/(\alpha + \beta)}$$
we consider the cases
$$ \mbox{(i)} \quad |\lambda-1/(pp')|^{1/2} \leq \eta  \qquad \text{and} \qquad \mbox{(ii)} \quad  |\lambda-1/(pp')|^{1/2} > \eta,$$ 
respectively. Note that  by (\ref{eq:11})
\begin{equation}
 \lambda - 1/(pp')=-z(z+2\gamma_p).\label{eq:47}
\end{equation}

Case (i):  Since for $z \in \mc_+$ we have $|z+2\gamma_p| \geq |z|$ we obtain
$$ \eta \geq |\lambda-1/(pp')|^{1/2}= |z(z+2\gamma_p)|^{1/2} \geq |z|.$$
Now we apply (\ref{eq:46}) with $a=\eta$, the sum being restricted to those $\lambda$ satisfying the first case, and use the estimate $|z+\eta| \leq |z| + \eta \leq 2\eta$. We obtain
{\small
  \begin{eqnarray}
   \sum_{|\lambda-1/(pp')|^{1/2} \leq \eta} \left( \frac{\dist(\lambda,\Sigma_p)}{|\frac 1 4 -\lambda|^{1/2}}\right)^{\delta_1} |z|^{\delta_2} &\leq& C \cdot C_1^r \cdot \eta^{\delta_1+\delta_2+\delta_3}(\eta+\gamma_p)^{r} \nonumber \\
&\leq& C \cdot C_1^{r+\frac{\delta_1+\delta_2+\delta_3}{\alpha+\beta}}(C_1^{\frac 1 {\alpha + \beta}}+\gamma_p)^{r}. \label{as}
  \end{eqnarray}
}
  \begin{rem}
    Note that here the constants $C$ are different from each other and from the one in
    (\ref{eq:46}), but they depend on the same parameters. Also in the
    following this constant may change from line to line.
  \end{rem}
It remains to estimate the sum on the left-hand side of the previous inequality from below in a suitable manner. To this end we note that since $1/4-1/(pp')=\gamma_p^2$ we have

\begin{equation}
  |\lambda-1/4|^{1/2} \leq \left( |\lambda-1/(pp')|+ \gamma_p^2 \right)^{1/2} \leq |\lambda-1/(pp')|^{1/2}+ \gamma_p. \label{eq:48}
\end{equation}
Moreover, this estimate implies that, with $z=\Psi_p^{-1}(\lambda)=-\gamma_p + \sqrt{1/4-\lambda}$,
\begin{equation}
|z+2\gamma_p| \leq |z+\gamma_p| + \gamma_p = |1/4 - \lambda|^{1/2} + \gamma_p \leq |\lambda-1/(pp')|^{1/2} + 2 \gamma_p.\label{eq:49}
\end{equation}

Finally, the previous inequality and (\ref{eq:47}) show that
\begin{equation}
|z| = \frac{|\lambda-1/(pp')|}{|z+2\gamma_p|} \geq \frac{|\lambda-1/(pp')|}{|\lambda-1/(pp')|^{1/2} + 2 \gamma_p}.\label{eq:50}
\end{equation}
\begin{rem}
  It is important to note that (\ref{eq:48})-(\ref{eq:50}) are valid for all $\lambda \in \Sigma_p^c$.
\end{rem}
Now we can use (\ref{eq:50}) and (\ref{eq:48}) to estimate the sum in (\ref{as}) from below by
\begin{equation*}
   \sum_{|\lambda-1/(pp')|^{1/2} \leq (2C_1)^{1/(\alpha+\beta)}} \dist(\lambda,\Sigma_p)^{\delta_1} \frac{|\lambda-\frac 1 {pp'}|^{\delta_2}}{(|\lambda- \frac 1 {pp'}|^{1/2}+2\gamma_p)^{\delta_1+\delta_2}}.
 \end{equation*}
This completes the proof of inequality (\ref{eq:34}).

Case (ii): For those $\lambda$ satisfying the second case we have
$$ \eta < |\lambda-1/(pp')|^{1/2} = |z(z+2\gamma_p)|^{1/2} \leq |z+2\gamma_p|.$$
Now we restrict the sum in (\ref{eq:46}) to those $\lambda$ satisfying the second case, multiply left- and right-hand side of (\ref{eq:46}) by $a^{\delta_1+\delta_4-\delta_3}(a+2\gamma_p)^{-r-1-\tau}$ and integrate $a$ from $\eta$ to  $\infty$.

Then as a result for the rhs we obtain
\begin{eqnarray}
 C \cdot C_1^r \cdot \int_\eta^\infty da \: (a+2\gamma_p)^{-1-\tau}  
=  C \cdot C_1^r \cdot \frac 1 \tau (\eta+2\gamma_p)^{-\tau}. \label{Aw}
\end{eqnarray}
Moreover, for the lhs we obtain 
{\small
  \begin{eqnarray}
 && \int_{\eta}^\infty da \: a^{\delta_1+\delta_4-\delta_3}(a+2\gamma_p)^{-r-1-\tau}
\sum_{|\lambda-1/(pp')|^{1/2}>\eta} \left( \frac{\dist(\lambda,\Sigma_p)}{|\frac 1 4 -\lambda|^{1/2}}\right)^{\delta_1} \frac{|z|^{\delta_2}}{|a+z|^{2\delta_1+\delta_2+\delta_4}} \nonumber\\
&\geq& \sum_{|\lambda-1/(pp')|^{1/2}>\eta} \left( \frac{\dist(\lambda,\Sigma_p)}{|\frac 1 4 -\lambda|^{1/2}}\right)^{\delta_1} |z|^{\delta_2} \int_{\eta}^\infty da \frac{a^{\delta_1+\delta_4-\delta_3}} {(a+2\gamma_p)^{r+1+\tau}(a+|z|)^{2\delta_1+\delta_2+\delta_4}}. \label{qw}
 \end{eqnarray}
}
Now we change variables in the integral in (\ref{qw}), obtaining that 
{\small
\begin{eqnarray}
&&  \int_{\eta}^\infty da \frac{a^{\delta_1+\delta_4-\delta_3}} {(a+2\gamma_p)^{r+1+\tau}(a+|z|)^{2\delta_1+\delta_2+\delta_4}} \nonumber\\
&=& |z+2\gamma_p|^{\delta_1+\delta_4-\delta_3+1}  \int_{\eta/|z+2\gamma_p|}^\infty db \frac{b^{\delta_1+\delta_4-\delta_3}} {(b|z+2\gamma_p|+2\gamma_p)^{r+1+\tau}(b|z+2\gamma_p|+|z|)^{2\delta_1+\delta_2+\delta_4}} \nonumber\\ 
&\geq& |z+2\gamma_p|^{-\delta_1-\delta_2-\delta_3-r-\tau}  \int_{1}^\infty db \frac{b^{\delta_1+\delta_4-\delta_3}} {(b+1)^{r+1+\tau+2\delta_1+\delta_2+\delta_4}}, \label{AAv}
\end{eqnarray}
}
where in the last step we used that $|z| \leq |z+2\gamma_p|$ and $2\gamma_p \leq |z+2\gamma_p|$ for $z \in \mc_+$, and that $\eta < |z+2\gamma_p|$ as had been shown above. From (\ref{AAv}), (\ref{qw}) and (\ref{Aw}) we obtain that
{\small
 \begin{equation}
   \label{eq:51}
 \sum_{|\lambda-1/(pp')|^{1/2}>\eta} \left( \frac{\dist(\lambda,\Sigma_p)}{|\frac 1 4 -\lambda|^{1/2}}\right)^{\delta_1} \frac{|z|^{\delta_2}}{|z+2\gamma_p|^{\delta_1+\delta_2+\delta_3+r+\tau}}  \leq C \cdot C_1^{r}(C_1^{1/(\alpha+\beta)}+\gamma_p)^{-\tau}.
 \end{equation} 
}
Finally, we use (\ref{eq:48})-(\ref{eq:50}) to estimate the left-hand side of (\ref{eq:51}) from below by
{\small
 \begin{equation*}
 \sum_{|\lambda-1/(pp')|^{1/2}>(2C_1)^{1/(\alpha+\beta)}}  \frac{\dist(\lambda,\Sigma_p)^{\delta_1} |\lambda-1/(pp')|^{\delta_2}}{(|\lambda-1/(pp')|^{1/2}+2\gamma_p)^{2\delta_1+2\delta_2 + \delta_3+r+\tau}} 
 \end{equation*} 
}

This shows that also inequality (\ref{eq:35}) is valid and concludes the proof of Theorem \ref{thm:61}.
 
\section{Proof of Theorem \ref{thm:3} and \ref{thm:4}} \label{sec:7}

In this final section we use Theorem \ref{thm:61} to prove Theorem \ref{thm:3} and \ref{thm:4}, starting with the former. We set $H_0=-\Delta_p$ and $H_p=-\Delta_p+V$ acting in $L_p(\mh), 1 < p < \infty$.

\subsection{Proof of Theorem \ref{thm:3}}

Let $r \geq 2$. Since Theorem \ref{thm:3} is obviously true if $\|V\|_r=0$, we can assume that this is not the case. Now we apply Theorem \ref{thm:61} with the $l_r$-ideal $\mathcal S_r(L_2(\mh))$ (see Appendix \ref{ap:ideals} and Example \ref{ex:schatt}). By Corollary \ref{cor:est}  we have
$$ \|V(H_0-\Psi_2(z))^{-1}\|_{\mathcal S_r} \leq \|V\|_r \Re(z)^{-(1-1/r)}|z|^{-(1-2/r)}, \quad z \in \mc_+.$$
Moreover, Lemma \ref{lem:2} shows that for $a>0$
$$ \|(H_0-\Psi_2(a))^{-1}\|_{2,2} \leq a^{-2}.$$
Hence we can apply Theorem \ref{thm:61} with $Z_0=H_0, M=V, p=2$ and
$$ C_1 = \|V\|_r, \quad \alpha = 1-1/r, \quad \beta = 1-2/r, \quad \gamma=2,$$
and so $1/(\alpha+\beta)= r/(2r-3)$ and
\begin{eqnarray*}
  \delta_1 = r+\tau, \qquad \delta_2 = (r-3+\tau)_+, \qquad \delta_3 = 3-3r.
\end{eqnarray*}
Then (\ref{eq:34})  implies, using that $\gamma_2=0$, 
{\small
  \begin{eqnarray*}
\sum_{|\lambda-1/4|^{r-3/2} \leq (2\|V\|_r)^{r}} \frac{\dist(\lambda,[1/4,\infty))^{r+\tau}}{|\frac 1 4 -\lambda|^{\frac{r+\tau-(r-3+\tau)_+}{2}}} \leq C \cdot \|V\|_r^{\frac r {2r-3} ( r+\tau +(r-3+\tau)_+)}.
\end{eqnarray*}
}
In particular, if we restrict to $\tau \in (0,1)$ and consider the cases $2 \leq r \leq 3 -\tau$ and $r>3-\tau$ separately, the validity of Theorem \ref{thm:3}, part (ia) and (ib), is easily derived.

Similarly, (\ref{eq:35}) implies that 
\begin{eqnarray*}
\sum_{|\lambda-1/4|^{r-3/2} > (2\|V\|_r)^{r}} \frac{\dist(\lambda,[1/4,\infty))^{r+\tau}}{|\frac 1 4 -\lambda|^{(3+3\tau)/2}}
\leq C'  \|V\|_r^{\frac r {2r-3}(2r-3-\tau)}.
\end{eqnarray*}
This shows the validity of Theorem \ref{thm:3}, part (ii), and concludes the proof of the Theorem.

\subsection{Proof of Theorem \ref{thm:4}}

In view of Proposition \ref{prop:1} it is sufficient to prove the theorem in case $p>2$. Let $r \geq p > 2$ and $\|V\|_r \neq 0$ (otherwise the theorem is trivially satisfied). As remarked in Appendix \ref{ap:ideals}, Example \ref{ex:summ}, if $r \geq  p > 2$ the $(r,p)$-summing ideal $\Pi_{r,p}(L_p(\mh))$ is an $l_R$-ideal, where 
\begin{equation}
  \label{eq:52}
  R = r + \eps(r) \quad \text{ and } \quad \eps(r)= \left\{
    \begin{array}{cl}
      0, & \text{if } r=p \\
      \eps_0, & \text{if } r>p .
    \end{array}\right.
\end{equation} 
Here $\eps_0 > 0$ can be chosen arbitrarily small. By Corollary \ref{cor:est}  we have
$$ \|V(H_0-\Psi_p(z))^{-1}\|_{\Pi_{r,p}} \leq 2^{1-2/p}\|V\|_r \Re(z)^{-(1+2/p-3/r)}, \quad z \in \mc_+,$$
and Lemma \ref{lem:2} shows that for $a>0$
$$ \|(H_0-\Psi_p(a))^{-1}\|_{p,p} \leq a^{-2}.$$
Hence we can apply Theorem \ref{thm:61} with the $l_R$-ideal $\Pi_{r,p}, Z_0=H_0, M=V$ and 
$$ C_1 = 2^{1-2/p}\|V\|_r, \quad \alpha = 1+2/p-3/r, \quad \beta = 0, \quad \gamma = 2,$$
so that $1/(\alpha+\beta)= r/(r(1+2/p)-3)$ and for $\tau \in (0,1)$
\begin{eqnarray*}
  \delta_1 &=& R(1+2/p-3/r)+1+\tau \\
  &=& r(1+2/p)-2+\tau + \eps(r)(1+2/p-3/r), \\
  \delta_2 &=& (-1+\tau)_+=0, \\
 \delta_3 &=& R(1-(1+2/p-3/r)-2)=R(-2-2/p+3/r) \\
 &=& 3-r(2+2/p)-\eps(r)(2+2/p-3/r).
\end{eqnarray*}
Before applying (\ref{eq:34}) in the present situation, we note that for $|\lambda-\frac 1 {pp'}|^{\frac 1 2} \leq (2C_1)^{1/(\alpha+\beta)}$ we trivially have
$ (|\lambda- 1/(pp')|^{\frac 1 2}+2\gamma_p) \leq (2C_1)^{1/(\alpha+\beta)} + 2\gamma_p$
and hence (\ref{eq:34}) implies that 
\begin{eqnarray*}
&&     \sum_{|\lambda- \frac 1 {pp'}|^{\frac 1 2} \leq (2C_1)^{\frac 1 {\alpha+\beta}}}  \dist(\lambda,\Sigma_p)^{\delta_1} \cdot |\lambda-\frac 1 {pp'}|^{\delta_2} \\
&\leq& C \cdot C_1^{r+\frac{\delta_1+\delta_2+\delta_3}{\alpha+\beta}}(C_1^{\frac 1 {\alpha + \beta}}+\gamma_p)^{r+\delta_1+\delta_2}\\
&\leq& C \cdot C_1^{r}(C_1^{\frac 1 {\alpha + \beta}}+\gamma_p)^{r+2\delta_1+2\delta_2+\delta_3}.
\end{eqnarray*}

Inserting the parameters computed above, the previous estimate and a short computation shows that with 
\begin{equation*}
 \eps_1:= \tau+\eps(r)(1+2/p-3/r), \qquad \eps_2:= 2\tau +\eps(r)(2/p-3/r),
\end{equation*}
we have
{\small
\begin{eqnarray*}
  &&   \sum_{|\lambda- \frac 1 {pp'}|^{\frac 1 2} \leq (2\|V\|_r)^{r/(r(1+2/p)-3)}}  \dist(\lambda,\Sigma_p)^{r(1+2/p)-2+\eps_1} \\
&\leq& C \cdot \|V\|_r^{r} (\|V\|_r^{r/(r(1+2/p)-3)}+\gamma_p)^{r(2/p+1)-1+\eps_2}.
\end{eqnarray*}
}
Note that choosing $\eps(r)$ sufficiently small we can achieve that $0< \eps_1, \eps_2 < 4\tau$. Since 
\begin{equation}
k=r(2-2\gamma_p)-2=r(1+\frac 2 p)-2 \quad \text{if} \quad  p > 2,\label{eq:53}
\end{equation}
this concludes the proof of Theorem \ref{thm:4}, part (i).

Similarly, considering 'large' eigenvalues we first note that from (\ref{eq:35}) we obtain, using that here $\delta_2=0$,
\begin{eqnarray*}
   \sum_{|\lambda- \frac 1 {pp'}|^{\frac 1 2}>(2C_1)^{\frac 1 {\alpha+\beta}}}  \frac{\dist(\lambda,\Sigma_p)^{\delta_1} }{(|\lambda-\frac 1 {pp'}|^{\frac 1 2}+2\gamma_p)^{2\delta_1+ \delta_3+r+\tau}} 
\leq C' \cdot C_1^{r}(C_1^{\frac 1 {\alpha+\beta}}+\gamma_p)^{-\tau}.
\end{eqnarray*}
Inserting the parameters this shows that, with
$$ \eps_3:= 3\tau + \eps(r)(2/p-3/r)$$
and $k$ as in (\ref{eq:53}) we have
\begin{eqnarray*}
\sum_{|\lambda- \frac 1 {pp'}|^{\frac 1 2} > (2\|V\|_r)^{r/(k-1)}}  \frac{\dist(\lambda,\Sigma_p)^{k+\eps_1} }{(|\lambda-\frac 1 {pp'}|^{\frac 1 2}+2\gamma_p)^{k+1+\eps_3}} \\
\leq C' \cdot \|V\|_r^{r}(\|V\|_r^{r/(k-1)}+\gamma_p)^{-\tau}.
\end{eqnarray*}
Since we can choose $\eps(r)$ sufficiently small such that $0<\eps_1,\eps_3 < 4 \tau$, this shows the validity of part (ii) of Theorem \ref{thm:4} and concludes the proof of the theorem.

\appendix
\section*{Appendix}
\setcounter{section}{1} 
\subsection{Operators, spectra and perturbations}\label{ap:oper}

We introduce terminology and collect some standard results on operators and spectra. As references see, e.g., \cite{kato, b_Gohberg69, MR1130394}.  

(a) $X$ and $Y$ denote complex Banach spaces and $\bdd(X,Y)$ denotes the algebra of all bounded linear operators from $X$ to $Y$. As usual we set $\bdd(X):=\bdd(X,X)$. The spectrum of a closed operator $Z$ in $X$ will be denoted by $\sigma(Z)$ and $\rho(Z):=\mc \setminus \sigma(Z)$ denotes its resolvent set. An isolated eigenvalue $\lambda$ of $Z$ will be called discrete if its algebraic multiplicity  $m(\lambda):=\dim(\operatorname{Ran}(P_Z(\lambda))$ is finite. Here 
$$ P_Z(\lambda)=\frac 1 {2\pi i} \int_\gamma (\mu-Z)^{-1} d\mu$$
denotes the Riesz-Projection of $Z$ with respect to $\lambda$ (and $\gamma$ is a counterclockwise oriented circle centered at $\lambda$, with sufficiently small radius). The set of all discrete eigenvalues is called the discrete spectrum $\sigma_d(Z)$. The essential spectrum $\sigma_{ess}(Z)$ is defined as the set of all $\lambda \in \mc$, where $\lambda-Z$ is not a Fredholm operator. We have $\sigma_{ess}(Z) \cap \sigma_d(Z) =   \emptyset$ and if $\Omega \subset \mc \setminus \sigma_{ess}(Z)$ is a connected component and $\Omega \cap \rho(Z) \neq \emptyset$, then $\Omega \cap \sigma(Z) \subset \sigma_d(Z)$. Moreover, each point on the topological boundary of $\sigma(Z)$ either is a discrete eigenvalue or a point in the essential spectrum. The discrete eigenvalues of $Z$ can accumulate at the essential spectrum only. Finally, the spectral mapping theorem for the resolvent says that for $a \in \rho(Z)$ we have
$$ \sigma((Z-a)^{-1}) \setminus \{0\} = \{ (\lambda-a)^{-1}: \lambda \in \sigma(Z)\}.$$
A similar identity holds for the essential and the discrete spectra as well. In the latter case, the algebraic multiplicities of $\lambda \in \sigma_d(Z)$ and $(\lambda-a)^{-1} \in \sigma_d((Z-a)^{-1})$ coincide.

(b) In this paper the sum $Z+M$ of two closed operators $Z,M$ in $X$ will always denote the usual operator sum defined on $\dom(Z) \cap \dom(M)$ (and the product $ZM$ is defined on $\{ f \in \dom(M): Mf \in \dom(Z)\}$). The operator $M$ is called $Z$-compact if $\dom(Z)\subset \dom(M)$ and $M(Z-a)^{-1}$ is compact for one (hence all) $a \in \rho(Z)$. If this is the case the sum $Z+M$ is closed and for $a \in \rho(Z+M) \cap \rho(Z)$ also the resolvent difference
$$ (Z+M-a)^{-1} - (Z-a)^{-1} = -(Z+M-a)^{-1}M(Z-a)^{-1}$$
is compact. In particular, Weyl's theorem on the invariance of the essential spectrum under compact perturbations and the spectral mapping theorem imply that $\sigma_{ess}(Z)=\sigma_{ess}(Z+M)$.

(c) If $Z$ is closed and densely defined $Z^*$ denotes its adjoint (see \cite[Sections III.5.5 and III.6.6]{kato}). The spectrum $\sigma(Z^*)$ is the mirror image of $\sigma(Z)$ with respect to the real axis and $[(\lambda-Z)^{-1}]^*=(\overline \lambda - Z^*)^{-1}$. Moreover, $\lambda \in \sigma_d(Z)$ iff $\overline{\lambda} \in \sigma_d(Z^*)$ and the respective algebraic multiplicities coincide. Finally, we note that if $M$ is another operator in $X$ and $ZM$ is densely defined, then $(ZM)^* \supset M^*Z^*,$ with equality if $Z \in \bdd(X)$.

\subsection{$l_r$-ideals and perturbation determinants}\label{ap:ideals}

We recall some results concerning the construction of perturbation determinants on Banach spaces. The main reference is \cite{Hansmann15}, see also \cite{MR889455, MR917067}.

Let $X$ denote a complex Banach space and let $r>0$. A quasi normed subspace $(\mathcal I, \|.\|_{\mathcal I})$ of $\bdd(X)$ is called an \emph{$l_r$-ideal} (in $\bdd(X)$) with \emph{eigenvalue constant} $\mu_r>0$ if the following holds:
\begin{enumerate}
    \item[(1)] The finite rank operators, denoted by $ \mathcal F(X)$, are dense in $\mathcal I$.
  \item[(2)] $\|L\| \leq \|L\|_{\mathcal I}$ for all $L \in \mathcal I$.
  \item[(3)] If $L \in \mathcal I$ and $A,B \in \bdd(X)$, then $ALB \in \mathcal I$ and $$\|ALB\|_{\mathcal I} \leq \|A\| \|L\|_{\mathcal I} \|B\|.$$
 \item[(4)] For every $L \in \mathcal I$ one has $\|(\lambda_j(L))\|_{l_r} \leq \mu_r \|L\|_{\mathcal I}$. Here $(\lambda_j(L))$ denotes the sequence of discrete eigenvalues of $L$, counted according to their algebraic multiplicity (note that by (1) and (2) each $L \in \mathcal I$ is compact).
\end{enumerate}

In the present paper we will need only two particular $l_r$-ideals, which we introduce in the following two examples.
\begin{ex}\label{ex:schatt}
Let $\hil$ denote a complex Hilbert space and let $r>0$. The \emph{Schatten-von Neumann classes} $\mathcal S_r(\hil)$ are defined by
$$ \mathcal S_r(\hil) = \{ K \in \bdd(X) : K \text{ is compact and } (s_n(K)) \in l_r\}.$$
Here $(s_n(K))$ denotes the sequence of singular values of $K$. Equipped with the (quasi-) norm
$ \|K\|_{\mathcal S_r} := \|(s_n(K))\|_{l_r}$ this class is an $l_r$-ideal with eigenvalue constant $\mu_r=1$.
\end{ex}

\begin{ex} \label{ex:summ}
Let $1 \leq q \leq p < \infty$. An operator $L \in \bdd(X)$ is called \emph{$(p,q)$-summing} if there exists $\rho>0$ such that for all finite systems of elements $x_1,\ldots,x_n \in X$ one has
$$ \left( \sum_{k=1}^n \|Lx_k\|^p \right)^{1/p} \leq \rho \sup_{x' \in X', \|x'\|\leq 1} \left( \sum_{k=1}^n |x'(x_k)|^q\right)^{1/q}.$$
We denote the infimum of all such $\rho>0$ by $\|L\|_{\Pi_{p,q}}$ and the class of all such operators by $\Pi_{p,q}(X)$.  In the special case $p=q$ we speak of $p$-summing operators and write $\Pi_{p}(X)$. We note that for $1\leq q_1 \leq q_0 \leq p_0 \leq p_1 < \infty$ we have $\Pi_{p_0,q_0}(X) \subset \Pi_{p_1,q_1}(X)$ and 
\begin{equation}
\|L\|_{\Pi_{p_1,q_1}} \leq \|L\|_{\Pi_{p_0,q_0}}, \qquad L \in \Pi_{p_0,q_0}(X).\label{eq:54}
\end{equation}

Moreover, if $\hil$ is a complex Hilbert space, then for $r \geq 2$ we have $\Pi_{r,2}(\hil)=\mathcal S_r(\hil)$ and the corresponding norms coincide. Concerning the above properties of an $l_r$-ideal we note that $\Pi_{p,q}(X)$ always satisfies (2) and (3), and it satisfies (1) if $X'$ has the approximation property and is reflexive, see \cite[Remark 5.4]{Hansmann15}. For such $X$, we can use known information on the eigenvalue distribution of the $(p,q)$-summing operators to make the following statements:
 
(a) $(\Pi_p(X),\|.\|_{\Pi_p})$ is an $l_{\max(p,2)}$-ideal with eigenvalue constant $\mu_{\max(p,2)} = 1$. 
  
(b) If $p>2$, then the eigenvalues of $L \in \Pi_{p,2}(X)$ are in the weak space $l_{p,\infty}(\mn)$, see \cite{MR576647}. More precisely, if the eigenvalues are denoted decreasingly $|\lambda_1(L)| \geq |\lambda_2(L)| \geq \ldots$ (where each eigenvalue is counted according to its algebraic multiplicity and the sequence is extended by $0$ if there are only finitely many eigenvalues) then
$$ \sup_{j \in \mn} |\lambda_j(L)| j^{1/p} \leq 2e \|L\|_{\Pi_{p,2}}.$$
In particular, this implies that for $q>p$ and $n \in \mn$
$$ \sum_{j=1}^n |\lambda_j(L)|^q = \sum_{j=1}^n \left(|\lambda_j(L)| j^{1/p}\right)^q j^{-q/p} \leq \left( 2e \|L\|_{\Pi_{p,2}}\right)^q \sum_{j=1}^\infty j^{-q/p}.$$
Hence we see that  $\Pi_{p,2}(X)$ is an $l_q$-ideal for every $q>p>2$, with eigenvalue constant $\mu_q^q= (2e)^q \sum_{j=1}^\infty j^{-q/p}$. Moreover, by (\ref{eq:54}) we see that for $p > r \geq 2$ also $\Pi_{p,r}(X)$ is an $l_q$-ideal for every $q>p$, with the same constant $\mu_q$ as before. 
\end{ex}

The $l_r$-ideals can be used to construct perturbation determinants on Banach spaces: First, for a finite rank operator $F \in \mathcal F(X)$ and $r>0$ we define
\begin{equation}
  \label{eq:55}
  {\det}_r(I-F):= \prod_j \left( (1-\lambda_j(F))\exp\left( \sum_{k=1}^{\lceil r \rceil-1} \frac{\lambda_j^k(F)}{k} \right) \right).
\end{equation}
Now one can show that for every $l_r$-ideal $(\mathcal I, \|.\|_{\mathcal I})$ there exists a unique continuous function ${\det}_{r,\mathcal I}(I-.) : (\mathcal I,\|.\|_{\mathcal I}) \to \mc$ that coincides with $\det_r(I-.)$ on the finite rank operators $\mathcal F(X)$. Moreover, there exists $\Gamma_r > 0$ such that for all $L \in \mathcal I$ we have
$$ |{\det}_{r,\mathcal I}(I-L)| \leq \exp\left( \mu_r^r \Gamma_r \|L\|_{\mathcal I}^r\right),$$
where $\mu_r$ denotes the eigenvalue constant of $\mathcal I$. Finally, if $A \in \bdd(X)$ and $K \in \mathcal I$, we define the \emph{$r$-regularized perturbation determinant} $d$ of $A$ by $K$ (with respect to $\mathcal I$) as follows:
$$ d: \rho(A) \ni \lambda \mapsto {\det}_{r,\mathcal I}(I-K(\lambda-A)^{-1}).$$
Then the following holds: (i) $d$ is analytic on $\rho(A)$. (ii) $\lim_{|\lambda| \to \infty} d(\lambda)=1.$ (iii) For $\lambda \in \rho(A)$ we have 
$$ |d(\lambda)| \leq \exp\left( \mu_r^r \Gamma_r \|K(\lambda-A)^{-1}\|_{\mathcal I}^r\right).$$
(iv) $d(\lambda)=0$ iff $\lambda \in \sigma(A+K)$. (v) If $\lambda \in \rho(A) \cap \sigma_d(A+K)$, then its algebraic multiplicity as an eigenvalue of $A+K$ coincides with its order as a zero of $d$.
 
\subsection{Complex Interpolation}\label{ap:inter} We review some aspects of Calderon's method of complex interpolation, see \cite{MR0167830} or \cite{MR0482275}.

Let $ S := \{ z \in \mc : 0 \leq \Re(z) \leq 1\}$ and let $(X,Y)$ denote an interpolation couple of complex Banach spaces (i.e. $X$ and $Y$ are complex Banach spaces continuously embedded in a topological vector space $V$). Then $X \cap Y$ and $X+ Y$ become Banach spaces when equipped with the norms $\|z\|_{X \cap Y} = \max(\|z\|_X,\|z\|_Y)$ and $\|z\|_{X+Y}=\inf\{ \|x\|_X + \|y\|_Y : z = x+y, x \in X, y \in Y\}$, respectively. We denote by $\mathcal G(X,Y)$ the vector space of all functions $f: S \to X+Y$ which satisfy the following properties:
\begin{itemize}
    \item[-] $f$ is holomorphic in the interior of $S$,
    \item[-] $f \in C_b(S;X+Y)$, i.e. $f$ is continuous and bounded on $S$,
    \item[-] $t \mapsto f(it) \in C_b(\mr;X)$ and $t \mapsto f(1+it) \in C_b(\mr;Y)$.
\end{itemize}
Then $\mathcal G(X,Y)$ becomes a Banach space with the norm
$$ \|f\|_{\mathcal G(X,Y)}:= \max( \sup_{t \in \mr} \|f(it)\|_X, \sup_{t \in \mr} \|f(1+it)\|_Y).$$
For $0< \theta < 1$ the \emph{complex interpolation spaces} $[X,Y]_\theta$ are introduced as follows:
$$ [X,Y]_\theta = \{ f(\theta) : f \in \mathcal G(X,Y)\}, \quad \|z\|_{[X,Y]_\theta} = \inf_{f \in \mathcal G(X,Y), f(\theta)=z} \|f\|_{\mathcal G(X,Y)}.$$
One can show that 
$$ X \cap Y \subset [X,Y]_\theta \subset X+Y,$$
both embeddings being continuous. 
\begin{ex}
 For a $\sigma$-finite measure space $(M,\mu)$ and $p_0,p_1 \in [1,\infty]$ we have
$$ [L_{p_0}(M),L_{p_1}(M)]_\theta = L_p(M), \quad \text{where} \quad \frac 1 p = \frac {1-\theta} {p_0} + \frac \theta {p_1}.$$
Moreover, the corresponding norms coincide.
\end{ex}
\begin{prop}\label{prop:inter}
  Let $f \in \mathcal G(X,Y)$ and set 
$$A_0= \sup_{t \in \mr} \|f(it)\|_X \quad \text{and} \quad A_1=\sup_{t \in \mr} \|f(1+it)\|_Y.$$
Then $\|f(\theta)\|_{[X,Y]_\theta} \leq A_0^{1-\theta} A_1^\theta.$ 
\end{prop}
\begin{proof}
 If $A_0, A_1 \neq 0$, the function $ g(w):= ({A_0}/{A_1})^{w-\theta} f(w), w \in S,$
is in $\mathcal G(X,Y)$ with $g(\theta)=f(\theta)$ and
$$ \sup_{t \in \mr} \|g(it)\|_X \leq A_0^{1-\theta} A_1^\theta \quad \text{and} \quad \sup_{t \in \mr} \|g(1+it)\|_Y \leq A_0^{1-\theta} A_1^\theta.$$
Hence $ \|f(\theta)\|_{[X,Y]_\theta} \leq \|g(\theta)\|_{\mathcal G(X,Y)} \leq A_0^{1-\theta} A_1^\theta.$ If one of $A_0,A_1$ vanishes, we can replace it by $\eps > 0$ in the definition of $g$ and then send $\eps \to 0$. If $A_0=A_1=0$, we can choose $g(w)=0$ to obtain the result.
\end{proof}
In this paper we will not need interpolation results for operators between abstract interpolation spaces. However, we will need the following more concrete result known as the \emph{Stein interpolation theorem} \cite{MR0082586} (see also \cite{MR0304972}).
\begin{rem}
  Let us recall that a \emph{simple function} on a measure space
  $(M,\mu)$ is a finite linear combination of characteristic functions
  of measurable sets of finite measure.
\end{rem}

In the following we denote the norm of $L_p(M)$ by $\|.\|_p$ and the operatornorm of $T:L_p \to L_q$ by $\|T\|_{p,q}$.

\begin{thm}\label{thm:stein}
Let $(M,\mu)$ and $(N,\nu)$ be $\sigma$-finite measure spaces and assume that for every $w \in S$, $T_w$ is a linear operator mapping the space of simple functions on $M$ into measurable functions on $N$. Moreover, suppose that for all simple functions $f: M \to \mc$ and $g : N \to \mc$, the product $T_wf \cdot g$ is integrable and that
$$ S \ni w \mapsto \int_N (T_wf)(x)g(x) \nu(dx)$$
is continuous and bounded on $S$ and holomorphic in the interior of $S$. Finally, suppose that for some $p_j,q_j \in [1,\infty], j=0,1,$ and $A_0,A_1 \geq 0$ we have
$$ \|T_{it}f\|_{{q_0}} \leq A_0 \|f\|_{{p_0}}, \quad \|T_{1+it}f\|_{{q_1}} \leq A_1 \|f\|_{{p_1}}$$
for all $t \in \mr$ and all simple functions $f: M \to \mc$. Then for each $\theta \in (0,1)$ and
$$ 1 / {p_\theta} = {(1-\theta)}/ {p_0} + {\theta} /{p_1}, \quad 1/ {q_\theta} =  {(1-\theta)}/ {q_0} + {\theta}/ {q_1},$$
the operator $T_\theta$ can be extended to a bounded operator in $\bdd(L_{p_\theta}(M),L_{q_\theta}(N))$ and
$$ \|T_\theta\|_{{p_\theta},{q_\theta}} \leq A_0^{1-\theta} A_1^\theta.$$
\end{thm}

 \bibliography{Bibliography} 
 \bibliographystyle{plain}

\end{document}